\documentclass[11pt, a4paper,reqno]{amsart}
\usepackage{amsfonts, amssymb, amscd, amsmath, amsthm}
\usepackage{setspace}
\usepackage[utf8]{inputenc}
\usepackage{graphicx}
 \usepackage{float}
\usepackage{fullpage}
\allowdisplaybreaks 
\usepackage{mathtools}
\newtheorem{thm}{Theorem}[section]
\newtheorem{lemma}[thm]{Lemma}
\newtheorem{prop}[thm]{Proposition}

\newtheorem{defn}[thm]{Definition}
\newtheorem{eg}[thm]{Example}

\newtheorem{Question}[thm]{Question}

\newtheorem{problem}[thm]{Problem}
\newtheorem{rmk}[thm]{Remark}

\newcommand{\A}{\mathcal{A}}

 \usepackage[bookmarksnumbered, colorlinks, plainpages]{hyperref}
  \hypersetup{colorlinks=true,linkcolor=blue, anchorcolor=green, citecolor=cyan, urlcolor=blue, filecolor=magenta, pdftoolbar=true}

\usepackage[T1]{fontenc}%
\allowdisplaybreaks
\usepackage{comment}
\title{Block decompositions of operator moment dilations}
\author{B. V. Rajarama Bhat}
\address{Statistics and Mathematics Unit, Indian Statistical Institute, R. V. College Post, Bangalore 560059, India}
\email{bhat@isibang.ac.in}
\author{Anindya Ghatak}
\address{Bennett University,Tech Zone 2, Greater Noida, Uttar Pradesh 201310, India}
\email{anindya.ghatak123@gmail.com}
\author{Santhosh Kumar Pamula}
\address{Department of Mathematical Sciences, Indian Institute of Science Education and Research (IISER) Mohali, SAS Nagar, Manauli Post, Punjab 140306, India}
\email{santhoshkp@iisermohali.ac.in} \subjclass[2010]{46L07, 47A12,
47A20, 47A57, 47B35} \keywords{Moment problem, Dilation, Block
operator, Poisson transform, Positive-definite kernel, Toeplitz
operator, Hankel operator}
\date{\today}
\begin{document}
\maketitle

%

%
%
%
%
%
%
%
%




\date{January 1, 2004}

\begin{abstract}
 Let  $\big\{A_{n}\big\}_{n\geq 1}$ be a sequence of bounded linear operators on a complex Hilbert space
   $\mathcal{H}$.  Then a bounded operator $B$ on a  Hilbert space $\mathcal{K} \supseteq \mathcal{H}$ is said to
   be a dilation of this sequence if
   \begin{equation*}
       A_{n} = P_{\mathcal{H}}B^{n}|_{\mathcal{H}} \; \text{for all}\; n\geq 1,
   \end{equation*}
where $P_{\mathcal{H}}$ is the projection of $\mathcal{K}$ onto
$\mathcal{H}.$ The question of the existence of dilation  is a
generalization of the classical moment problem. We obtain the necessary
and sufficient conditions for the existence of self-adjoint
dilation. Then  we provide explicit block operator
representations for these dilations. For example, the self-adjoint dilations can be put in block tri-diagonal form.

Given a positive invertible operator $A$, an operator $T$ is said to
be in the $\mathcal{C}_{A}$-class if the sequence $\{
A^{-\frac{1}{2}}T^nA^{-\frac{1}{2}}\}_{n\geq 1}$  admits a unitary
dilation. We identify a tractable collection of
$\mathcal{C}_A$-class operators for which isometric and unitary
dilations can be written explicitly in block operator form.
This includes the well-known $\mathcal{C}_{\rho}$-class for positive scalars $\rho$.
Here the special cases $\rho =1$ and $\rho =2$ correspond to
Sch\"{a}ffer representation for contractions and Ando's representation
for operators with a numerical radius not more than one respectively.
\end{abstract}

\maketitle
\section{Introduction}
Starting from the pioneering work of  Sz.\@-Nagy (\cite{SF10}),  the
dilation theory has played  a fundamental role in  operator theory.
This has been explored by many researchers. The early work mostly used
techniques from classical function theory and basic operator theory.
Thanks to W. Arveson, V. I. Paulsen, and  others (\cite{Paulsen 2002})
we have new tools coming from the theory of completely positive
maps. The literature here is vast, and so we simply refer to a recent
survey  of dilation theory by Orr Shalit \cite{Shalit 21} and the
references therein.

Let $\mathcal{H}$ be a complex Hilbert space and let
$\big\{A_{n}\big\}_{n\geq 1}$ be a sequence of bounded linear
operators on
   $\mathcal{H}$.  Then a bounded operator $B$
   on a  Hilbert space $\mathcal{K} \supseteq \mathcal{H}$   is said to
  be  a dilation if
   \begin{equation*}
       A_{n} = P_{\mathcal{H}}B^{n}|_{\mathcal{H}} \; \text{for all}\; n\geq 1.
   \end{equation*}
 Generally, it is convenient to include $n=0$ in the sequence, where we would
be taking $A_0=I$, the identity operator of the Hilbert space. The
operator moment problem is to determine conditions on the sequence
of operators to ensure the existence of a dilation $B$ with prescribed
property, for instance, we may require $B$ to be
unitary/isometry/self-adjoint/positive  or the
spectrum of $B$ to be contained in a given subset of complex plane.

Two special cases of this problem are very well known.  A 
result of Sz.-Nagy tells us that if $A_n=A^n, n\geq 0$ for some
operator $A$, then it admits a unitary dilation if and only if $A$
is a contraction.

Another special situation is when ${\mathcal H}$ is one-dimensional
so that $A_n$ is a sequence of scalars and $B$ is a self-adjoint
operator on ${\mathcal K}$. This amounts to requiring
$$A_n= \langle u, B^nu\rangle = \int _{\sigma (B)}x^nd\mu(x)$$
where $u$ is a unit vector in ${\mathcal K}$ and $\mu $ is the
spectral measure of $B$ coming from the vector state $\langle u,
(\cdot )u\rangle.$ In other words, this is the classical moment
problem of seeking a measure with specified moments.

The general operator moment problem is less well known, but if we
start searching, we find considerable literature scattered here and
there. This problem was already thought of by Sz.-Nagy (\cite{Nagy
1952}). Some further references can be found in \cite{Bisgard 1994,
Luminita 95, Vasilescu 09}).
 Also, we \textcolor{black}{refer to the} excellent book of V. I. Paulsen \cite{Paulsen 2002},
 where the author discusses the  (isometric/ unitary, positive) dilation problem for operator-valued sequences.

The initial objective of this article is to present a cohesive
summary of fundamental findings related to the operator moment
problem. We hope that this helps us to provide
some historical context to the results presented here. This is not
intended to be a survey.
Here in we mostly use modern
techniques  coming from the theory of completely positive maps or
we use the method of positive kernels to show the existence of the
dilation. The main aspect that we emphasize is the existence of dilations as block operator matrices, which can be obtained using the Sch\"{a}ffer type construction for such matrices.
 This seems to be a fairly general phenomenon.
This is very useful because it helps us  visualize the dilation. For
the powers of a contraction, the Sch\"{a}ffer construction gives the
dilation as a $2\times 2$ block operator perturbation of the
bilateral shift. In the last Section, we obtain $4\times 4$ block
operator perturbations of the bilateral shift appearing as dilation
operators.

 Throughout the article, $\mathcal{H}$ denotes a complex Hilbert space and we follow
\textcolor{black}{Physicists' convention} by considering the inner product $\langle
\cdot , \cdot \rangle$ on $\mathcal{H}$ as linear in the second
variable and anti-linear in the first variable. Also note that
$\mathcal{B}(\mathcal{H})$ denotes the space of all bounded linear
operators on $\mathcal{H}$ and $C(X)$ is the $C^{\ast}$-algebra of
all continuous functions on a compact Hausdorff space $X.$
 In general, we denote $C^{\ast}$-algebras by $\mathcal{A}, \mathcal{B}$ etc. In particular, the $C^{\ast}$ algebra of all $n\times n$
matrices with entries from $\mathcal{A}$ is denoted by
$M_{n}(\mathcal{A})$. For a subset $M$ of $\mathcal{H}$, the
subspace $[M]:= \overline{\rm span}(M)$ is the smallest closed
subspace of $\mathcal{H}$ that contains $M$.


 \begin{defn}[Dilation of operator sequences]
Let  $\{A_{n}\}_ {n\geq 0}$ be a sequence of bounded linear
operators on a Hilbert space $\mathcal{H}$ with $A_0=I$. The
operator sequence $\{A_{n}\}_ {n\geq 0}$ is said to admit a
\emph{dilation} if there \textcolor{black}{exists} a bounded linear operator $B$
 on a Hilbert space $\mathcal{K}\supseteq \mathcal{H}$ such that
\begin{equation}\label{Equation: Momentdilation}
    A_{n} = P_{\mathcal{H}} B^{n}\big|_{\mathcal{H}},\; \text{~for all~}\; n\geq 0.\;
\end{equation}
Then $B$ is called the dilation.  A dilation is said to be
positive/self-adjoint/ isometry/unitary if $B$ has that property.
A positive/self-adjoint/isometric dilation is said to be minimal if
\begin{equation}\label{Equation: Minimality} {\mathcal
K}=\overline{\rm span}\{B^n(\mathcal{H}): n\in {\mathbb
Z_+}\}.\end{equation}

\end{defn}

Some clarifications are in order. Here, by convention, for any
operator $B$, $B^0$ is taken as identity. So, condition
$A_0=P_{\mathcal{H}}B^0\big|_{\mathcal{H}}$ is superfluous once we
take $A_0=I.$ Therefore we are effectively just considering the
dilation of $\{A_n\}_{n\in {\mathbb N}}$. However, for minimality, it
is important to include $n=0$ in Equation \eqref{Equation:
Minimality}. In the case of unitary dilations, the minimality
condition in Equation \eqref{Equation: Minimality} should be replaced by
\begin{equation}\label{Equation: Unitary Minimality}
{\mathcal K}=\overline{\rm span}\{B^n(\mathcal{H}): n\in {\mathbb
Z}\}.\end{equation}

\begin{problem}\label{main problem}
Given a sequence of operators we wish to consider dilations with the
prescribed property of being positive, self-adjoint, etc. Some
natural questions that arise are the following:
\begin{enumerate}

\item \textbf{Existence:} What are necessary and sufficient conditions for existence of
dilations with prescribed property?

 \item \textbf{Uniquenss:} When a dilation with prescribed property exists, is it possible to prove uniqueness of the dilation up to unitary
equivalence  under the assumption of minimality?

\item  \textbf{Construction:} Can we explicitly construct these dilations instead of simply abstractly proving their existence and uniqueness.
\end{enumerate}
\end{problem}

The question of uniqueness is easy to answer. For easy reference,
we state it as a theorem.
\begin{thm}
Let $\{A_n\}_{n\geq 0}$ be a sequence of bounded operators on a
Hilbert space $\mathcal{H}$ admitting a
self-adjoint/positive/isometric/unitary dilation $B$ on a Hilbert
space $\mathcal{K}\supseteq \mathcal{H}.$ Then the given operator
sequence admits a minimal dilation with the same prescribed
property. Moreover, such a minimal dilation is unique up to unitary
equivalence.
\end{thm}

\begin{proof}
Suppose $B$ is a self-adjoint dilation on a Hilbert
space $\mathcal{K}\supseteq \mathcal{H}$. Then by restricting $B$ to
the subspace $\overline{\mbox{span}}\{B^n h: h\in \mathcal{H}, n\in
\mathbb{Z}_+\}$ we get a minimal self-adjoint dilation of
$\{A_{n}\}_{n \geq 0}$ as,
$$
\langle B^mg, B^nh\rangle  = \langle g, B^{m+n} h \rangle =  \langle g, A_{m+n} h \rangle \text{ for all }g,h \in \mathcal{H}, m,
n\in \mathbb{Z}_+
$$
If we consider two minimal dilations $B_{i}$ on $\mathcal{K}_{i}\;
(i = 1,2.)$, then we get an isometry by defining $U\big( B_{1}^{n}
h\big) = B_{2}^{n} h$ for every $h \in \mathcal{H},\; n \in
\mathbb{Z}_{+}$ and extending linearly and continuously.  It is
actually a unitary because of the minimality of the dilations.   Clearly,
the same statement holds for positive dilations.

Now assume that $B \in \mathcal{B}(\mathcal{K})$ is an isometric
dilation of the given sequence $\{A_{n}\}_{n\geq 0}$ in some
Hilbert space $\mathcal{K}$ containing $\mathcal{H}$. Since $B$ is
an isometry, the inner products  $\langle B^mg, B^nh\rangle$ are
determined as,
$$
\langle B^m g, B^n h \rangle = \left\{ \begin{array}{cc} \big\langle g, A_{n-m} h \big\rangle & \; \text{if} \; n \geq m\\
 & \\
 \big\langle g, \big(A_{m-n}\big)^{\ast} h \big\rangle & \; \text{if} \; n \leq m\end{array}\right.
$$
It follows that $B$ restricted to the subspace
 $\overline{\mbox{span}}\{B^n h: h\in \mathcal{H}, n\in\mathbb{Z}_+\}$ is the minimal isometric dilation of $\{A_{n}\}_{n\geq 0}$. The uniqueness of minimal isometric dilation upto unitary equivalence follows as in the earlier case. For unitary dilations we
replace $\mathbb{Z}_+ $ by  $\mathbb{Z}$ and we have the analogous
result.

\end{proof}

The first question is more delicate. It is obvious that  given an
arbitrary operator-valued sequence $\{A_{n}\}_{n \geq 0},$ there may
not exist $B$ satisfying equation \eqref{Equation:
Momentdilation}.
 Some necessary conditions are easily followed. We
list the following:
\begin{enumerate}
\item The sequence $\{A_n\}_{n\in {\mathbb Z}_+}$ should satisfy the growth bound $\|A_n\|\leq M^n$ for some $M>0.$
\item  For $B$ to be positive (resp. self-adjoint), $A_n$'s
should be positive (resp. self-adjoint). For $B$ to be isometric or
unitary, $A_n$'s should be contractive.
\end{enumerate}
The condition $(1)$ is natural as we are looking for dilations which
are bounded operators. The condition $(2)$ is also obvious.
It is worth to note that the question of existence of positive dilation and unitary/ Isometric dilation is \textcolor{black}{discussed in} \cite{Paulsen 2002} by V. I. Paulsen. Here, we discuss the necessary and sufficient conditions for self-adjoint dilation.

 We now briefly describe the plan of the article. We recall and
summarize some known answers to the first question of  Problem
\ref{main problem}  in Theorem
\ref{thm-self adjoint dilation 2},  and Theorem \ref{Thm: CA class equivalent}.
 To be more precise, in
Theorem \ref{thm-self adjoint dilation 2}, we have the necessary and
sufficient criteria for the self-adjoint
 dilation problem in terms of Hankel matrices. \textcolor{black}{Existence} of unitary/ Isometric dilation is given by V. I. Paulsen \cite{Paulsen 2002}. Consequently, Theorem \ref{thm-self adjoint dilation 2} \textcolor{black}{provides a} complete answer to  \textcolor{black}{question $1$} of Problem \ref{main problem}.
 Moreover, we examine the $\mathcal{C}_{A}$ class of operators as a particular instance of unitary dilation. In our study, we derive a necessary and sufficient condition for an operator-valued sequence to be a member of the $\mathcal{C}_{A}$ class (refer to Theorem \ref{Thm: CA class equivalent}).

Hence, the main focus of this article is the third question of {\it Problem} \ref{main problem}. Here, we try to obtain block operator forms for various classes of dilations. In
Theorem \ref{Theorem: Schaffer self-adjoint}, we show that if
$\{A_{n}\}_{n \geq 0}$ admits self-adjoint dilation (say $B$) then
$B$ has a tri-diagonal form (see Theorem \ref{Thm: tri diagonal})
whose blocks are given by the recursive relation (see Section 2). In Theorem \ref{Theorem: isometry}, we show that if $\{A_{n}\}_{n \geq
0}$ admits isometric dilation (say $V$), then $V$ has the form as in
Equation \eqref{Equation: Schaffer} and the blocks of $V$ are given by
the recursive relation described in Section 3.

 Then we focus on   operator-valued sequences of so-called  $\mathcal{C}_{A}$-class that admit isometric dilations.
 We present several necessary and sufficient conditions for an operator \textcolor{black}{to belong }$\mathcal{C}_{A}$-class
 (see  Theorem \ref{Thm: CA class equivalent} and Theorem \ref{Theorem: C_A class and CP map}).

In the final section, a special  subclass of
$\mathcal{C}_{A}$-class  operators for which we can write down
isometric and unitary dilations explicitly in block operator form
has been studied (see Theorem \ref{thm: explicit isometric
A-dilation}, and Theorem \ref{thm: explicit unitary A-dilation}).
Moreover, we describe their minimal dilation spaces (see Proposition
\ref{prop: explicit minimal dilation space} and Remark \ref{rmk: explicit minimal dilation space}).
The cases where $\rho=1$ and $\rho=2$ are particularly notable, as they correspond to the Sch\"{a}ffer representation for contractions and Ando's representation for operators with a numerical radius no greater than one, respectively.
In the process of investigation, we find \textcolor{black}{a notable} observation in Lemma \ref{basic lemma} \textcolor{black}{namely that every $\mathcal{C}_{\rho}$-class admits} partial isometric dilation. Moreover, we have provided an explicit method for accomplishing this.
\section{Self-adjoint dilations}

\subsection{Classical moment problems}

The moment problem  appeared for the first time in the pioneering
work of Stieltjes \cite{Stieltjes} in 1894. Before that P. L.
Chebyshev \cite{Chebyshev} and Markov \cite{Markov} used the notion
of moments in their work ``limiting value of integrals''. A survey
of these developments can be found in the article Kre\u{i}n
\cite{Krein}. However, Stieltjes was the first mathematician to
consider the moment problem as a  problem in its own right and
formulated it as: For a given sequence $\{m_{n}:\; n \geq 0\}$ of
real numbers, does there exists a  Radon measure $\mu$ supported on
$\mathbb{R}$ such that
\begin{equation}
    m_{n} = \int_{0}^{\infty} x^{n}\; d\sigma(x), ~~\forall\;  n\geq 0 ?
\end{equation}
Later in 1920, Hamburger \cite{Hamburger} extended the scope of the
problem by considering the existence of such measures  supported on
$[0, \infty)$ and Hausdorff \cite{Hausdorff} discussed the
solvability of the problem with existence of such measures supported
on $[0, 1].$ Several authors have made significant contributions in
this direction, for instance  see the work of R. Nevanlinna
\cite{Nevanlinna}, M. Riesz \cite{Riesz}, T. Carleman
\cite{Carleman}, and M.H. Stone \cite{Stone 1932}.

In general, for any closed subset $K$ of $\mathbb{R}$, the $K$- moment problem  \cite{Schmudgen 91} asks for the existence of a measure $\mu$ supported on
the set $K$ such that
\begin{equation}\label{Moment sequence: Integral from}
    m_{n} = \int_{K} x^{n}\; d\mu(x), ~~\forall n\geq 0.
\end{equation}
Three specific choices of $K$ stand out due to their natural importance and for historical reasons, namely the  \emph{Hamburger moment problem} (when $K = [0, \infty)$), the \emph{Stieltjes moment problem}  (when $K = \mathbb{R}$) and the \emph{Hausdorff moment problem} (when $K = [0, 1]$).

We refer the reader to  excellent surveys on the moment problem and
related problems in analysis given by Akhiezer  \cite{Akhiezer 1965}
and Schmudgen \cite{Schmudgen 2017}.  We recall some of the
necessary and sufficient conditions for the existence of solutions
of these moment problems from \cite{Akhiezer 1965, Schmudgen 2017}
(also see references therein).
 Firstly, for each $n \geq 0$, the associated Hankel matrices of the moment sequence $\{m_{n}\}_{n \geq 0}$ are defined by
\begin{align*}
H_{n}&= \begin{bmatrix}
m_{0}&m_{1}& \cdots & m_{n}\\m_{1}&m_{2}&\cdots &m_{n+1}\\ \vdots & \vdots & \ddots & \vdots\\m_{n}&m_{n+1}&\cdots&m_{2n}
\end{bmatrix}_{(n+1) \times (n+1)}\; \text{and}
&\\
H^{(1)}_{n}& =  \begin{bmatrix}
m_{1}&m_{2}& \cdots & m_{n+1}\\m_{2}&m_{3}&\cdots &m_{n+2}\\ \vdots & \vdots & \ddots & \vdots\\m_{n+1}&m_{n+2}&\cdots&m_{2n+1}
\end{bmatrix}_{(n+1) \times (n+1)}.
\end{align*}
The Hamburger moment problem has a solution \cite[Theorem
3.8]{Schmudgen 2017} if and only if the associated Hankel matrices
\begin{align*}
H_{n} \geq 0 \text{~for every~} n \geq 0.
\end{align*}
The Stieltjes moment problem has a solution \cite[Theorem 3.12]{Schmudgen 2017} if and only if the associated Hankel matrices
\begin{align*}
H_{n} \geq 0 \text{~and~} H_{n}^{(1)} \geq 0 \text{~for every~} n \geq 0.
\end{align*}
Furthermore, the Hausdorff moment problem has a solution \cite[Theorem 3.15]{Schmudgen 2017} if and only if the sequence $\{m_{n}:\; n\geq 0\}$ is completely monotonic i.e., $$(-1)^{k} (\Delta^{k}m)_{n} \geq 0 \text{~for every } n,k \geq 0,$$ where
\begin{equation*}
    (\Delta^{k}m)_{n} = \sum\limits_{i=0}^{k}{k\choose i} (-1)^{i} m_{i+n}.
\end{equation*}
Also, see \cite[Theorem 2.6.4]{Akhiezer 1965} for a detailed solution
of the Hausdorff moment problem. One may also consider measures
supported on subsets of the complex plane, for instance, see
\cite{Ahron 75, Schmudgen 2017}.

Coming to dilations of operator sequences, clearly the starting
point is the following celebrated theorem of Sz.-Nagy.
 Let $T\in \mathcal{B(H)}$  and consider the operator valued sequence $\{T^n\}_{n\geq 0}$. This sequence
  admits a minimal isometric or unitary dilation if and only if $T$ is a contraction. Moreover, the minimal dilations are
 unique up to unitary equivalence.  ( See \cite{SF10} or  \cite[Theorem 1.1]{Paulsen
 2002}). We call them dilations of $T$.  In 1955, Sch\"{a}ffer \cite{Schaffer 1955} provided an explicit construction of minimal isometric and unitary dilations
   as follows: Let $D_T=(I-T^*T)^{\frac{1}{2}}, D_{T^*}=(I-TT^*)^{\frac{1}{2}}$ and  $\mathcal{D}_{T}: = \overline{\rm range}D_T,$
   $\mathcal{D}_{T^*}= \overline{\rm range}D_{T^*}$.  Take
    $$\mathcal{K}:= \mathcal{H} \oplus \mathcal{D}_{T} \oplus \mathcal{D}_{T}\oplus\cdots
    .$$ Then
 \[ V=
\begin{bmatrix}
T &0&0&\cdots\\
D_T&0&0&\cdots\\
0&I&0& \cdots\\
0&0&I&\cdots\\
\vdots & \vdots&\ddots&\ddots
\end{bmatrix}
\]
on $\mathcal{K}$ is a minimal isometric dilation of $T$. Take
$\mathcal{L}= \cdots \oplus \mathcal{D}_{T^*}\oplus
\mathcal{D}_{T^*}\oplus \mathcal{H}\oplus \mathcal{D}_T\oplus
\mathcal{D}_T\oplus \cdots $ and define $U$ on $\mathcal{L}$ by
$$U= \left(%
\begin{array}{ccccccc}
    &  &  &  &  &  & \\
  \ddots  &  &  &  &  &  &  \\
   & I &  &  &  &  &  \\
     &  & I &  &  &  &  \\
     &  &  & D_{T^*} & {\bf T} &  &  \\
    &  &  & -T^* & D_T &  &  \\
  &&&&&I&\\
   &  &  &  &  &  & \ddots
\end{array} \right) . $$
The bold font indicates the location of operator $T$ from
$\mathcal{H}$ to $\mathcal{H}$ and $0$ entries are not displayed.
In this article we provide Sch\"{a}ffer type constructions for
several operator valued moment sequences.

\subsection{Self-adjoint dilations}

In 1952, Sz-Nagy obtained the following necessary and sufficient
condition for a sequence of operators to admit self-adjoint dilation
with spectrum contained in a given compact set.
  \begin{thm} \cite{Nagy 1952} \label{Theorem: Nagy1952}
 Let $X\subseteq \mathbb{R}$ be a compact set.  Let $\{A_{n}:\; n\geq 0\}$ be a sequence of
 bounded self-adjoint operators on a Hilbert space $\mathcal{H}$ with $A_{0} = I$. It admits a self-adjoint operator dilation
 $B$ with $\sigma(B)\subseteq X$ if and only if
 \begin{equation}\label{Equation: polynomialcondition}
     c_{0}+c_{1}A_{1} + \cdots + c_{n}A_{n} \geq 0, \;
 \end{equation}
whenever the complex polynomial $ c_{0}+c_{1}x+\cdots+c_{n}x^{n}
\geq 0\;$ for all
     $x \in  X.$
 \end{thm}

 The existence of self-adjoint dilation for an operator sequence is  linked with positivity of the corresponding Hankel matrix. To
 see this we begin with the following observation.
 \begin{lemma} \label{Lemma: positivetype}Let $\{A_{n}\}_{n\geq 0}$ be a
 sequence in $\mathcal{B}(\mathcal{H})$ with $A_{0} = I$. Then $\sum\limits_{i,j = 0}^{n} X_{i}^{\ast} A_{i+j} X_{j} \geq 0$ for any $X_{0}, X_{1}, \cdots , X_{n} \in \mathcal{B}(\mathcal{H}), n \geq 0$ if and only if the
 associated Hankel matrix
 \begin{equation*}
     H_{n}:= \begin{bmatrix} I & A_{1} & A_{2} & \cdots & A_{n}\\ A_{1}&A_{2}&A_{3} &\cdots &A_{n+1}\\\vdots& \vdots& \vdots & \cdots & \vdots\\ A_{n}&A_{n+1} & A_{n+2} & \cdots & A_{2n} \end{bmatrix} \geq 0, \;  \; \text{for all} \;\; n \geq 0.
 \end{equation*}
  \end{lemma}
  \begin{proof}
   Suppose that $H_{n} \geq 0$ for all $n \geq 0$ and let $X_{0}, X_{1}, \cdots, X_{n} \in \mathcal{B}(\mathcal{H})$. For every $g \in \mathcal{H}$, we see that
   \begin{align*}
       \big\langle g,\; \Big(\sum\limits_{i,j = 0}^{n} X_{i}^{\ast} A_{i+j} X_{j}\Big) g \big\rangle
       &= \sum\limits_{i,j = 0}^{n} \langle X_{i}{g},\; A_{i+j} X_{j}{g} \rangle = \langle \widetilde{g}, H_{n} \widetilde{g} \rangle \geq 0,
   \end{align*}
   where $\widetilde{g} = \begin{bmatrix} X_{0}g\\\vdots\\ X_{n}g\end{bmatrix}$. To prove the converse,
 choose and fix  $g \in \mathcal{H}$ with $\|g\| =1.$
Now  for  $h_{0}, h_{1}, \cdots, h_{n} \in \mathcal{H}$, take
$X_i=|h_{i}\rangle \langle g|\; (0\leq i \leq n)$. Then
   \begin{align*}
       \Big\langle \begin{bmatrix} h_{0} \\ \vdots\\ h_{n}\end{bmatrix}, \; H_{n} \begin{bmatrix} h_{0} \\ \vdots\\ h_{n}\end{bmatrix} \Big\rangle
       & = \sum\limits_{i,j = 0}^{n} \langle h_{i},\; A_{i+j} h_{j} \rangle \\
       &= \Big\langle g, \; \Big(\sum\limits_{i,j=0}^{n}|h_{i}\rangle \langle g|^{\ast} A_{i+j} |h_{j}\rangle \langle g| \Big) g \Big\rangle\\
       &= \langle g, \sum _{i,j=0}^nX_i^*A_{i+j}X_jg\rangle \\
       &\geq 0.
   \end{align*}
  \end{proof}
Let $\{A_{n}\}_{n\geq 0}$ be a sequence of self-adjoint operators
acting on a Hilbert space $\mathcal{H}$. Consider the associated
Hankel matrices:
$$ H_{n}:= \begin{bmatrix} I & A_{1} & A_{2} & \cdots & A_{n}\\
A_{1}&A_{2}&A_{3} &\cdots &A_{n+1}\\\vdots& \vdots& \vdots & \cdots
& \vdots\\ A_{n}&A_{n+1} & A_{n+2} & \cdots & A_{2n} \end{bmatrix},$$ 
and 
$$
H^{(2)}_{n} := \begin{bmatrix} A_{2} & A_{3} & A_{4} & \cdots &
A_{n+2}\\ A_{3} & A_{4} & A_{5} & \cdots & A_{n+3}\\ \vdots & \vdots
& \vdots & \cdots & \vdots\\A_{n+2} & A_{n+3}& A_{n+4} & \cdots &
A_{2n+2}
 \end{bmatrix}.$$

In the following theorem we use both the method of completely
positive maps as well as that of positive kernels. Following Theorem
\ref{thm-self adjoint dilation 2} presents a new characterization of
bounded self-adjoint dilations in terms of Hankel matrices.

Now we recall a few definitions and basic results from
the theory of $C^{\ast}$-algebras. Let $\mathcal{A}$ and $
\mathcal{B}$ be unital $C^{\ast}$-algebras. It is easy to see that
$M_{n}(\mathcal{A})$, the collection of all matrices with entries
from $\mathcal{A}$ is a $C^{\ast}$-algebra and so is
$M_{n}(\mathcal{B}).$ A typical element in $M_{n}(\mathcal{A})$ is
denoted by $[a_{ij}]$. A linear map $\varphi \colon \mathcal{A} \to
\mathcal{B}$ is said to be:
\begin{enumerate}
\item {\it positive,} if $\varphi(a) \geq 0$ in $\mathcal{B}$ for every $a \geq 0$ in $\mathcal{A}$\\
\item {\it completely positive,} if for each $n \in \mathbb{N}$ the map $\varphi_{n}\colon M_{n}(\mathcal{A}) \to M_{n}(\mathcal{B})$ given by
\begin{equation*}
\varphi_{n}\Big( \big[ a_{ij}\big]\Big) := \big[ \varphi(a_{ij})\big]
\end{equation*}
is positive, for each $n \in \mathbb{N}.$\\

\item {\it completely bounded,} if $\big\| \varphi\big\|_{cb}:= \sup\big\{\big\|\varphi_{n}\big\|:\; n \in \mathbb{N} \big\} < \infty.$
\end{enumerate}
Note that a positive map need not be completely positive. For example, if we define
\begin{equation*}
\varphi\big(  [a_{ij}]\big)  = [a_{ji}] \;\; \text{for all}\;\; [a_{ij}] \in M_{n}(\mathbb{C})
\end{equation*}
then $\varphi$ is positive but not completely positive. Every completely positive map $\varphi \colon \mathcal{A} \to \mathcal{B}$ is completely bounded with  $\|\varphi\|_{cb} = \|\varphi(1_{\mathcal{A}})\|,$ where $1_{\mathcal{A}}$ is a unit element in $\mathcal{A}$ (see \cite[Proposition 3.6]{Paulsen 2002}) and it follows from Theorem 3.11 of \cite{Paulsen 2002} that if $\mathcal{A}$ is  commutative then $\varphi \colon \mathcal{A} \to \mathcal{B}$ is positive implies $\varphi$ is completely positive.  For a detailed discussion and insight on basic results related to completely positive and completely bounded maps, we refer the reader to \textcolor{black}{the  book} by Vern I. Paulsen \cite{Paulsen 2002}.

\begin{thm}\label{thm-self adjoint dilation 2} Let $\{A_{n}\}_{n\geq 0}$ be a sequence of self-adjoint
 operators with $A_0=I$ and  $\Vert A_{n}\Vert \leq 1$ for all $n$. Then
 it admits a
 self-adjoint contraction dilation  if and only if
 \begin{equation*}
 H_{n}\geq0  \text{ and } H^{(2)}_{n} \leq H_{n} \; \text{ for each }\; n.
\end{equation*}
\end{thm}
\begin{proof} Suppose there exists a self-adjoint contraction  dilation $B$ on a Hilbert space $\mathcal{K}\supseteq \mathcal{H}$.
 Since $\sigma(B)\subseteq [-1,1]$, then it follows from the Theorem \ref{Theorem: Nagy1952} that the map $\varphi \colon C[-1,1] \to \mathcal{B(H)}$ defined by
 \begin{equation*}
 \varphi(x^m) = A_{m}\; \text{ for}\; m \geq 0.
 \end{equation*}
 is positive and hence it is completely positive since the domain $C\big( [-1, 1]\big)$ is a commutative $C^{\ast}$-algebra. Now consider the element
$L\in M_{n+1} (C[-1,1])$, where
 \begin{equation*}
    L= \begin{bmatrix} 1 & x & x^2& \cdots & x^{n}\\0&0&0&\cdots &0\\ 0&0&0&\cdots&0\\\vdots&\vdots&\vdots& \cdots & \vdots \\0&0&0&\cdots&0  \end{bmatrix}.
\end{equation*}
Then $L^*L \geq 0$ in $M_{n+1} (C[-1,1])$. Since $H_{n}=
\varphi_{n+1}(L^*L),$ and $\varphi$ is completely positive,
$H_{n}\geq 0.$ We define $G\in M_{n+1} (C[-1,1])$ by

$$ G = \begin{bmatrix} \sqrt{1-x^2} & x\sqrt{1-x^{2}} &  \cdots & x^{n}\sqrt{1-x^{2}}\\0&0&\cdots &0\\ 0&0&\cdots&0\\\vdots&\vdots& \cdots & \vdots \\0&0&\cdots&0  \end{bmatrix}.$$
Then $G^*G \geq 0$ in $M_{n+1} (C[-1,1])$. Since $H_{n}-
H^{(2)}_{n}= \varphi_{n+1}(G^*G),$ and $\varphi$ is completely
positive,
$H_{n} - H^{(2)}_{n}   \geq 0.$

 Conversely, assume that $H_{n}\geq0$ and $H^{(2)}_{n} \leq H_{n}$ for each $n.$
 Let $M: = \mathbb{Z}_{+} \times \mathcal{H}$. Define a map $k \colon M \times M \to \mathbb{C}$ by
\begin{equation}\label{Equation: Hausdorffkernel}
    k((m,g), (n,h)) = \langle g,\; A_{m+n}h\rangle,\; \; \text
{for every}\; m,n \in \mathbb{Z}_{+},\; g,h \in \mathcal{H}.
\end{equation}
Since $H_{n}\geq 0$ for all $n\geq0,$ then by Lemma \ref{Lemma: positivetype}, it follows that $k$ is a positive definite
kernel. Let $V$ be a vector space of all complex functions on $M$
which is zero except for finitely many points of $M$. Since $k$ is
positive definite,  $V$ is a semi-inner product space with respect
to:
\begin{equation*}
    \langle \xi, \eta \rangle: = \sum\limits_{x,y \in M} \overline{\xi(x)} \eta(y) k(x, y)\; \text{for every}\; \xi, \eta \in V.
\end{equation*}
Let us take  $\mathcal{N} = \{\xi\in V:\; \langle \xi, \xi \rangle =
0\}$. Then by Cauchy-Schwarz inequality,  $\mathcal{N} = \{\xi:\;
\langle \xi, \eta \rangle = 0,\; \text{for every}\; \eta \in V\}$ is
subspace of $V$.  Take  $\mathcal{K}$ as the Hilbert space obtained
by the completion of the quotient space $V/\mathcal{N}$. Define
$\lambda\colon M \to \mathcal{K}$ by
\begin{equation*}
    \lambda(m,g) = \delta_{(m,g)}+ \mathcal{N},\; \text{for all}\; (m,g) \in M.
\end{equation*}
Since $ \langle\lambda(0,g), \lambda(0,h)\rangle =  \langle g,
h\rangle $ for every $g, h  \in \mathcal{H}$, we see that
$\mathcal{H}$ can be identified as a subspace of $\mathcal{K}$ via
the map  $g \mapsto \lambda(0,g).$ Moreover, $$\mathcal{K} =
\overline{\rm span} \{\lambda (m,g):\; m\geq 0,\;  g \in
\mathcal{H}\}$$. Define $B (\lambda (m,g)) = \lambda(m+1, g)$ for
every $(m,g) \in M$. Without loss of generality,
$m_{0}=0, m_{1} = 1, \cdots, m_{n} = n$, $g_{i} \in \mathcal{H}$ for
$0\leq i \leq n$ and since $H_{n}^{(2)} \leq H_{n},$ we get
\begin{align*}
\Big\| B\big(\sum\limits_{i=0}^{n}{c_{i}} \lambda(m_{i}, g_{i}) \big) \Big\|^{2} &=\sum\limits_{i,j=0}^{n}\overline{c_{i}}c_{j}\langle \lambda(i+1, g_{i}),\; \lambda(j+1, g_{j}) \rangle\\
&= \sum\limits_{i,j=0}^{n}\overline{c_{i}}c_{j} \langle g_{i}, \; A_{i+j+2} (g_{j}) \rangle \\
&=  \Big\langle \begin{bmatrix} c_{0}g_{0} \\ \vdots\\ c_{n}g_{n}\end{bmatrix},\; H^{(2)}_{n}\begin{bmatrix} c_{0}g_{0} \\ \vdots\\ c_{n}g_{n}\end{bmatrix} \Big\rangle\\
&\leq \Big\langle \begin{bmatrix} c_{0}g_{0} \\ \vdots\\ c_{n}g_{n}\end{bmatrix},\; H_{n}\begin{bmatrix} c_{0}g_{0} \\ \vdots\\ c_{n}g_{n}\end{bmatrix} \Big\rangle\\
& \leq \sum\limits_{i,j=0}^{n}\overline{c_{i}}c_{j} \langle g_{i}, \; A_{{i}+{j}} (g_{j})
\rangle \\
&= \Big\|\sum\limits_{i=0}^{n}{c_{i}} \lambda(m_{i}, g_{i})
\Big\|^{2}.
\end{align*}
This implies that $B$ is a contraction on a dense subspace and it
extends to a linear contraction on $\mathcal{K}$. Denoting the
extension also  by $B$, it is easy to see that $B$ is self-adjoint.
Moreover,  for every $g,h \in \mathcal{H}$,
\begin{align*}
    \langle \lambda (0,g), P_{\mathcal{H}}B^{n}|_{\mathcal{H}}\lambda(0,h) \rangle
     = \langle \lambda(0,g), \; \lambda(n,h) \rangle= \langle g, A_{n} h \rangle.
\end{align*}
Therefore, $B$ is  a self-adjoint contraction dilation of $\{A_n\}_{n\geq 0}.$
\end{proof}
\begin{rmk} \label{Remark: A2biggerthanA1square}
If a sequence $\{A_n:n\geq 0\}$ with $A_0=I$ admits a self-adjoint dilation, then by Theorem \ref{thm-self adjoint dilation 2}, it follows that the $2\times 2$ Hankel matrix
\begin{equation*}
H_{2} = \begin{bmatrix}
I & A_{1}\\
A_{1} & A_{2}
\end{bmatrix} \geq 0.
\end{equation*}
Equivalently, $A_{2} \geq A_{1}^{2}.$  Alternatively, if a
self-adjoint operator $B$ on some Hilbert space $\mathcal{K}
\supseteq \mathcal{H}$, is a dilation,
\begin{equation}\label{Equation: A2-A1square}
A_{2} - A_{1}^{2} = P_{\mathcal{H}}B^{2}P_{\mathcal{H}} - P_{\mathcal{H}}BP_{\mathcal{H}}BP_{\mathcal{H}} = P_{\mathcal{H}}B (I - P_{\mathcal{H}}) BP_{\mathcal{H}} \geq 0.
\end{equation}
\end{rmk}
Let  $B$ be a bounded self-adjoint operator on some Hilbert space $\mathcal{K} \supseteq \mathcal{H}$ such that $A_{n} = P_{\mathcal{H}}B^{n}P_{\mathcal{H}}$ for every $n \geq 0.$  If we assume $A_2=A_1^2$ then from Equation \eqref{Equation: A2-A1square} we get $(I-P_{\mathcal{H}})BP_{\mathcal{H}} = 0.$ It means that the subspace $\mathcal{H}$ of $\mathcal{K}$ is reducing under $B.$ It follows that $A_{1} = B\big|_{\mathcal{H}}$ and so,
\begin{equation*}
A_{n} =  B^{n}\big|_{\mathcal{H}} = \big(B\big|_{\mathcal{H}}\big)^{n} = A_{1}^{n},\; \text{for}\; n \geq 1.
\end{equation*}

In fact, much stronger results are known now (see \cite{PS2021} for
more details). In the work of K. Schm\"{u}dgen \cite{Schmudgen 2017,
Schmudgen 91} \textcolor{black}{one finds a necessary} and sufficient conditions for
the existence of solution of $K$-moment problem where $K$ is a
compact basic semi-algebraic set (i.e., the solution set of a finite
system of polynomial inequalities $p_{1}(x)\geq 0, p_{2}(x) \geq 0,
\cdots, p_{n}(x)\geq 0$.  F.-H. Vasilescu \cite{Vasilescu 1998}
builds on Schm\"{u}dgen's results and gives an explicit solution of
$K$-moment problem in the case of semi-algebraic sets. For a
detailed discussion of these topics one can see the survey article
\cite{Vasilescu 03}.

\subsection{Concrete self-adjoint dilations}
Now we turn our discussion to a concrete construction of
self-adjoint dilation. Before that, let us recall some known facts
from the literature. Let us recall some known
facts from the literature connecting self-adjoint operators and
Jacobi/tri-diagonal matrices. Suppose  $B$ is a bounded self-adjoint
operator defined on a separable Hilbert space admitting a unit
cyclic vector $v$. It can be represented by a tri-diagonal matrix
with respect to the basis obtained by Gram-Schmidt process from the
set $\{v, Bv, B^{2}v, \cdots\}$. Suppose the matrix $B$ is expressed
as,
\begin{equation}\label{tri}
B =
\begin{bmatrix}
  a_{0} & \overline{b_{0}} & 0 & 0& \cdots  \\
   {b_{0}} & a_{1} & \overline{b_{1}} &0& \hdots \\
    0      & {b_{1}} & a_{2} &\overline{b_{2}} & \hdots  \\
    0      & 0      & {b_{2}} & a_{3}& \ddots \\
    \vdots & \vdots & \ddots & \ddots & \ddots \\
\end{bmatrix}.\end{equation}
Then it is related to a family of monic orthogonal polynomials given
by $p_0(x)=1, p_1(x)=(x-a_0)$,
\begin{equation}\label{Equation: orthogonal polynomial}
    p_{n}(x) = (x-a_{n-1})p_{n-1}(x) - |b_{n-1}|^{2}p_{n-2}(x), \; \text{for all}\; n\geq 2
\end{equation}
Note that here equation \eqref{Equation: orthogonal polynomial} is
obtained by expanding the determinant of $(xI-B_{n})$ using the last
row, where $B_{n}$ is the $n\times n$ truncated matrix of $B$. This
is being cited  from \cite[Lemma 3.2]{BP}. It is to be noted that
the cyclicity of $v$, ensures that $b_j\neq 0$ for every $j$.
 By replacing the basis vectors by the \textcolor{black}{same vectors}  multiplied by   suitable phase-factors one may actually take $b_j>0$ for every $j$.
 The recurrence
relations require minor modifications if we consider orthonormal
polynomials instead of orthogonal polynomials. Exact formulae and
other basic information about this theory can be found in the
classic book \cite{Stone 1932}.  Some applications of this idea to
quantum theory with worked out examples can be seen in \cite{BP}.

\textcolor{black}{Let $\mu(\cdot) = \|E(\cdot)v\|^{2}$ be the probability measure
defined on the Borel $\sigma-$field of the real line, where $E$ is
the spectral measure associated with $B$. Clearly, it is supported on
the spectrum of $B$. Then $\{p_{n}(x):\; n \geq 0\}$ is an
orthonormal basis of $L^{2}(\mu)$ (for more details, see \cite{Stone
1932}). Moreover, $m$-th moment of $\mu$ is given by
\begin{align*}
 \langle v,\; B^{m}v \rangle = \int\limits_{-\infty}^{\infty} \lambda^{m}\; d\langle v,\; E_{B}(\lambda)v \rangle,\; \text{for all}\; m \geq 0.
\end{align*}
}
Conversely, given any compactly supported probability measure $\mu $
on the real line we can take $B$ as the operator `multiplication by
$x$', on $L^2(\mu )$ and \textcolor{black}{the cyclic vector} $v$ as the constant function
$1$. We can observe the tri-diagonal form of $B$ on the \textcolor{black}{basis of}
normalized orthogonal polynomials. The coefficients $\{a_n,
b_n:n\geq 0\}$ are known as Jacobi parameters of the measure $\mu .$
Here $B$ is a self-adjoint dilation of the moment sequence of the
probability measure.

 This motivates us to construct such tri-diagonal
operator matrix $B$ for a self-adjoint dilation of an operator
sequence $\{A_{n}:\; n \geq 0\}.$ Such tri-diagonal blocks, known as
generalized Jacobi tri-diagonal relations, often appear in quantum
theory (See \cite{AL} for details).

\begin{lemma}\label{Upper Hessenberg}
Let $B$ be a bounded operator on  some Hilbert space ${\mathcal K}$.
Let ${\mathcal H}$ be a closed subspace of ${\mathcal K}.$ Assume
\begin{equation}\label{minimality}
{\mathcal K}= \overline{\mbox span}\{ B^{n} h: h\in {\mathcal H},
n\in {\mathbb Z}_+\}.\end{equation} Then ${\mathcal K}$ decomposes
as a direct sum of Hilbert spaces
\begin{equation*}
{\mathcal K}={\mathcal H}_0\oplus {\mathcal H}_1\oplus {\mathcal H}_2\oplus \cdots
\end{equation*}
where ${\mathcal H}_0={\mathcal H}$ and with respect to this
decomposition the operator $B$ has the `upper Hessenberg' form:
 $$B= \begin{bmatrix}\label{upper Hessian}
    B_{00} & B_{01} & B_{02} & B_{03} & \cdots  \\
    B_{10} & B_{11} & B_{12} & B_{13} & \cdots \\
    0      & B_{21} & B_{22} & B_{23}& \cdots  \\
    0      & 0 & B_{32} & B_{33} & \cdots \\
    \vdots & \vdots & \vdots & \ddots & \ddots \\
    \end{bmatrix} $$
where $\mathcal {H}_n =\overline{B_{n (n-1)}(\mathcal{H}_{n-1})}$
for every $n$. Conversely, any such $B$ as above satisfies Equation \eqref{minimality}.
\end{lemma}
\begin{proof} Take $${\mathcal H}_{n]}=\overline{\rm span} \{ B^mh:\; {0 \leq m \leq n}, h\in
{\mathcal H}\}$$ and
$$ \mathcal{H}_{0}:= \mathcal{H},\; \mathcal{H}_{1}:= \mathcal{H}_{1]} \bigcap \mathcal{H}^{\bot}_{0},\; {\mathcal H}_n= {\mathcal H}_{n]}\bigcap
{{\mathcal H}_{(n-1)]}^{\perp }}\text{ for } n\geq 2.$$
Then clearly
${\mathcal K}= \oplus_{n\geq
1} {\mathcal H}_n,$ and as $B ({\mathcal H}_{n]})\subseteq {\mathcal
H}_{(n+1)]}$, $B({\mathcal H}_n)\subseteq \oplus
_{m=0}^{n+1}{\mathcal H}_{m+1}$. Consequently, the operator $B$ has
the form described above. The range and condition and the converse
statements are easy to see.
\end{proof}

\begin{thm}\label{Thm: tri diagonal}
Let $\{A_{n}\}_{n\geq 0}$ be a sequence of self-adjoint operators in
$\mathcal{B}(\mathcal{H})$ with $A_0=I$, admitting a minimal
self-adjoint dilation $B$ in $\mathcal{B}(\mathcal{K})$ for some
Hilbert space $\mathcal{K}$. Then the space
$\mathcal{K}=\mathcal{H}_0\oplus \mathcal{H}_1\oplus \cdots$ (
$\mathcal{H}_0=\mathcal{H})$,  so that the operator $B$ has the
tri-diagonal form:
\begin{equation} \label{Equation: tri diagonal}
    B= \begin{bmatrix}
    B_{00} & B_{01}^{\ast} & 0 & 0 & \cdots  \\
    B_{10} & B_{11} & B_{21}^{\ast} & 0 & \cdots \\
    0      & B_{21} & B_{22} & B_{32}^{\ast}& \cdots  \\
    0      & 0      & B_{32} & B_{33} & \cdots \\
    \vdots & \vdots & \vdots & \ddots & \ddots \\
    \end{bmatrix},
\end{equation}
and $\mathcal{H}_n=\overline{B_{n(n-1)}(\mathcal{H}_{n-1})}$ for all
$n\geq 1.$
\end{thm}
\begin{proof}
From the previous lemma, $B$ has upper Hessenberg form. But since
$B$ is self-adjoint, $B_{ij}=B_{ji}^*=0$ for $j>(i+1).$

\end{proof}

Consider the setting up of this theorem.  We  try to determine the
blocks using the sequence $\{A_n\}_{n\geq 0}$ and the dilation
property. It is done recursively, and to do this we need to invert
some operators. In general, it is quite possible that some of these
operators are not invertible, and we may have to modify the
construction. For the moment, we assume invertibility of concerned
operators as and when required.

 Since $B$ is a self-adjoint dilation for the moment sequence
$\{A_{n}\}_{n \geq 0}$ we have $(B^{n})_{00} = A_{n}$ for every $n
\geq 0$. Clearly, $B_{00} = A_{1}$ and $B_{10}$ is obtained as,
\begin{align*}
    A_{2} = (B^{2})_{00}= B_{00}^{2}+B_{10}^{\ast}B_{10} = A_{1}^{2}+ B_{10}^{\ast}B_{10}.
\end{align*}
That is, $ B_{10}^{\ast}B_{10} = A_{2}-A_{1}^{2}$. Since
$(A_{2}-A_{1}^{2}) \geq 0$ (see Remark \ref{Remark:
A2biggerthanA1square}), we can choose $B_{10} =
(A_{2}-A_{1}^{2})^{1/2}.$ Firstly, we explain the process to compute
the diagonal block $B_{11}$ and then establish the recurrence
relation to obtain the diagonal and lower diagonal blocks. From the
Equation \eqref{Equation: tri diagonal}, we have
\begin{align*}
    A_{3} = (B^{3})_{00} = \sum\limits_{0\leq r_{1},r_{2}\leq 2} B_{0r_{1}}B_{r_{1}r_{2}}B_{r_{2}0} = \sum\limits_{0\leq r_{1},r_{2}\leq 1} B_{0r_{1}}B_{r_{1}r_{2}}B_{r_{2}0},
\end{align*}
since $B_{0r_{1}}B_{r_{1}r_{2}}B_{r_{2}0} = 0$ when either $r_{1}$
or $r_{2}$ is 2$.$ This implies that
\begin{align*}
A_{3} &= B_{00}^{3}+B_{00}B_{01}B_{10}+B_{01}B_{10}B_{00}+B_{01}B_{11}B_{10}\\
&= B_{00}^{3}+B_{00}B^{\ast}_{10}B_{10}+B^{\ast}_{10}B_{10}B_{00}+B^{\ast}_{10}B_{11}B_{10}.\\
\end{align*}
By substituting the first column information we get
\begin{equation*}
    B_{11} = (A_{2}-A_{1}^{2})^{-1/2} \Big[ (A_{3}-A_{1}^{3}) - A_{1}(A_{2}-A_{1}^{2}) - (A_{2}-A_{1}^{2})A_{1}\Big] (A_{2}-A_{1}^{2})^{-1/2}.
\end{equation*}
The formulae for diagonal and off-diagonal blocks are as follows.
Firstly, note that each diagonal block is a self-adjoint operator
and is computed by the compression of odd powers of $B$. Suppose
that  $(n-1)$ columns of $B$ are known then $B_{nn}$ is computed as
follows:
\begin{eqnarray*}
      A_{2n-1}
     &= & (B^{2n-1})_{00}\\
&= & \sum\limits_{0\leq r_{1}, r_{2}, \hdots, r_{2(n-1)} \leq 2n-2}B_{0r_{1}} B_{r_{1}r_{2}}\cdots B_{r_{2(n-1)}0}\\
&    = & \sum\limits_{\underset{ \& \; (r_{n-1}, r_{n}) \neq (n,n)}{0\leq r_{1}, r_{2}, \hdots, r_{2(n-1)}\leq 2n-2}} B_{0r_{1}} B_{r_{1}r_{2}} \cdots B_{r_{2(n-1)}0} \\
& & +\;  B_{01} \cdots B_{(n-1)n}B_{nn}B_{n(n-1)} \cdots B_{10}\\
 &   = &\sum\limits_{\underset{ \& \; (r_{n-1}, r_{n}) \neq (n,n)}{0\leq r_{1}, r_{2}, \hdots, r_{2(n-1)}\leq 2n-2}} B_{0r_{1}} B_{r_{1}r_{2}} \cdots B_{r_{2(n-1)}0} \\
 & & +\;  B_{10}^{\ast} \cdots B^{\ast}_{n(n-1)}B_{nn}B_{n(n-1)} \cdots B_{10}.
\end{eqnarray*}
This implies that
\begin{eqnarray*}
    B_{nn} = \Big( \prod\limits_{i=1}^{n}B^{\ast}_{i(i-1)}\Big)^{-1} E_{n} {\Big( \prod\limits_{i=1}^{n}B^{\ast}_{i(i-1)}\Big)^{\ast}}^{-1},
\end{eqnarray*}
where
\begin{equation*}
E_{n} = \Big[A_{2n-1} \; - \sum\limits_{\underset{ \& \; (r_{n-1}, r_{n}) \neq (n,n)}{1\leq r_{1}, r_{2}, \hdots, r_{2(n-1)}\leq 2n-2}} B_{1r_{1}} B_{r_{1}r_{2}} \cdots B_{r_{2(n-1)}1} \Big].
\end{equation*}

Now notice that the lower diagonal block in the first column is
given by $B_{21} = (A_{2}-A_{1}^{2})^{1/2}$. These lower diagonal
blocks can be obtained by the compression of even powers of $B.$
Suppose that  $(n-1)$ columns of $B$ are known then $B_{(n+1)n}$ is
obtained as below:
\begin{eqnarray*}
     A_{2n}&=&(B^{2n})_{00}\\
    &= &\sum\limits_{0\leq r_{1}, r_{2}, \hdots, r_{2n-1} \leq 2n-1}\!\!\!\!\! \!\!\!\!\! B_{1r_{1}} B_{r_{1}r_{2}}\cdots B_{r_{2n-1}0}\\
    &=& \sum\limits_{\underset{ \& \;  r_{n} \neq n}{0\leq r_{1}, r_{2}, \hdots, r_{2n-1}\leq 2n-1}} \!\!\!\!\! \!\!\!\!\! B_{0r_{1}} B_{r_{1}r_{2}} \cdots B_{r_{2n-1}0} \\
    & & +\; B_{01} \cdots B_{(n-2)(n-1)}B_{(n-1)n}B_{n(n-1)}B_{(n-1)(n-2)} \cdots B_{10}\\
    &=&\sum\limits_{\underset{ \& \; r_{n} \neq n}{0\leq r_{1}, r_{2}, \hdots, r_{2n-1}\leq 2n-1}} \!\!\!\!\!  \!\!\!\!\! B_{0r_{1}} B_{r_{1}r_{2}} \cdots B_{r_{2n-1}0} \\
    & & +\;  B_{10}^{\ast} \cdots B^{\ast}_{(n-1)(n-2)}B^{\ast}_{n(n-1)}B_{n(n-1)}B_{(n-1)(n-2)} \cdots B_{10}.
\end{eqnarray*}
This implies that
\begin{align*}
    & B^{\ast}_{n(n-1)}B_{n(n-1)}\\
    & = \Big( \prod\limits_{i=1}^{n-1}B^{\ast}_{i(i-1)}\Big)^{-1} \Big[A_{2n} \; - \sum\limits_{\underset{ \& \;
     r_{n}\neq n}{0\leq r_{1}, r_{2}, \hdots, r_{2n-1}\leq 2n-1}} \!\!\!\!\! \!\!\!\!\!  B_{0r_{1}} B_{r_{1}r_{2}} \cdots B_{r_{2n-1}0} \Big] \\
     & \;\;\;\;\;\;\;\;\;\;\;\;\;\;\;\;\;\;\;\; \;\;\;\;\;\;\;\;\;\;\;\;\;\;\;\;\;\;\;\; \;\;\;\;\;\;\;\;\;\;\;\;\;\;\;\;\;\;\;\;\;\;\;\;\;\;\;\;\;\;\;\;\;\;\;\;\;\;\;\;{\Big( \prod\limits_{i=1}^{n-1}B^{\ast}_{i(i-1)}\Big)^{\ast}}^{-1}.
\end{align*}
In this case, we can choose
\begin{equation*}
   B_{n(n-1)} = \Big\{\Big( \prod\limits_{i=1}^{n-1}B^{\ast}_{i(i-1)}\Big)^{-1} F_{n} {\Big( \prod\limits_{i=1}^{n-1}B^{\ast}_{i(i-1)}\Big)^{\ast}}^{-1}\Big\}^{1/2},
\end{equation*}
where
\begin{equation*}
F_{n} = \Big[A_{2n} \; - \sum\limits_{\underset{ \& \; r_{n} \neq n}{0\leq r_{1},
    r_{2}, \hdots, r_{2n-1}\leq 2n-1}}\!\!\!\!\! \!\!\!\!\! B_{0r_{1}} B_{r_{1}r_{2}} \cdots B_{r_{2n-1}0} \Big].
\end{equation*}
With these computations  and Lemma \ref{Upper Hessenberg}, we get
the following result.

\begin{thm}\label{Theorem: Schaffer self-adjoint}
Let $\{A_{n}\}_{n\geq 0}$ be a sequence of self-adjoint operators in
$\mathcal{B}(\mathcal{H})$ with $A_0=I$, admitting a self-adjoint
dilation. If the inverses appearing in the  recurrence relations
above are well-defined bounded operators, then these formulae
provide a self-adjoint minimal dilation with  block tri-diagonal form
as above.
\end{thm}
\begin{eg} Let $T\in \mathcal{B(H)}$ be a self-adjoint operator. We define a moment sequence $\{A_{n}\}_{n\geq 0}$  by
\begin{align*}
A_{n} = &\left\{
        \begin{array}{cc}
            2^{\frac{n-2}{2}}T^{n}   &\text{~if~} n \text{ is even}\\
            &                    \\
            0     &\text{~if~} n \text{ is odd}~~~~~~.   \\
        \end{array} \right.\\
    \end{align*}
Then  $V$ acting on $\mathcal{H}\oplus\mathcal{H}\oplus\mathcal{H}$,
\begin{align*}
V=
\begin{bmatrix}
0                &{T}     &0\\
T &0                     &T\\
0               & T     &0
\end{bmatrix}
\end{align*}
is a self-adjoint dilation of $\{A_n:n\geq 0\}.$
\end{eg}

This example can be generalized as follows. Suppose $\{m_n: n\geq
0\}$ is a moment sequence of a compactly supported probability
measure $\mu $ on $\mathbb{R}.$ Let $T$ be a bounded self-adjoint
operator on a Hilbert space $\mathcal{H}$. Then $\{m_nT^n: n\geq
0\}$ is an operator moment sequence that admits self-adjoint dilation.
Take the dilation space $\mathcal{K}=\mathcal{H}\otimes L^2(\mu )$
with $\mathcal{H}$ identified as a subspace by identifying $h\in
\mathcal{H}$ with $h\otimes 1$, where $1$ is the constant function 1
in $L^2(\mu ).$ Let $B$ be the tri-diagonal form of multiplication
by `$x$' operator on $L^2(\mu )$ with cyclic vector $1$ as in
Equation \eqref{tri}. Then  $T\otimes B$ is a self-adjoint dilation of
$\{m_nT^n: n\geq 0\}$:
\[T\otimes B = \begin{bmatrix}
  a_{0}T & b_{0}T & 0 & 0& \cdots  \\
   {b_{0}}T & a_{1}T & b_{1}T &0& \hdots \\
    0      & {b_{1}}T & a_{2}T &b_{2}T & \hdots  \\
    0      & 0      & {b_{2}}T & a_{3}T& \ddots \\
    \vdots & \vdots & \ddots & \ddots & \ddots \\
\end{bmatrix}.\]

An operator $T\in \mathcal{B(H)}$ is called \emph{quasinormal} if
$T(T^*T)= (T^*T)T.$ We  recall that $T$ is \emph{quasinormal} if and only
if $(T^*T)^k= T^{*k}T^k$ for all $k\in \mathbb{Z}_{+}$
\cite{PS2021}.
\begin{eg} Let $T\in \mathcal{B(H)}$ be a quasinormal operator. Then the moment sequence $\{A_{n}\}_{n\geq 0}$ defined by
\begin{align*}
A_{n} = &\left\{
        \begin{array}{cc}
            2^{\frac{n-2}{2}}T^{*\frac{n}{2}}T^{\frac{n}{2}}   &\text{~if~} n \text{ is even}\\
            &                    \\
            0     &\text{~if~} n \text{ is odd}~~~~~~   \\
        \end{array} \right.\\
    \end{align*} admits a self-adjoint dilation. In fact,
 $V$ acting on $\mathcal{H}\oplus\mathcal{H}\oplus\mathcal{H}$  defined by
\begin{align*}
V=
\begin{bmatrix}
0                &T^*     &0\\
T &0                     &T^*\\
0               & T     &0
\end{bmatrix}
\end{align*}
is a self-adjoint dilation of $\{A_n: n\geq 0\}.$
\end{eg}

Like before, this example  can also be generalized  to have
self-adjoint dilations for operator moment sequences $\{1, 0,
m_2T^*T, 0, m_4(T^*)^2T^2, 0, \ldots \}$, where $\{m_n:n\geq 0\}$ is the
moment sequence of a compactly supported probability measure $\mu $
on $\mathbb{R}$, which is symmetric around $0$. Such a symmetry
ensures that odd moments and also diagonal Jacobi parameters of the
measure are all equal to 0.



\section{Unitary and isometric dilations}
In this section, we discuss the necessary and sufficient
conditions  on operator  sequences to admit unitary  or isometric
dilations. Let $\{A_{n}\}_{n \geq 0}$ be a sequence in
$\mathcal{B}(\mathcal{H})$ with $A_{0}=I$.
 Then $\{A_{n}\}_{n\geq 0}$ is said to admit a unitary (respectively isometric) dilation if there is
  a Hilbert space $\mathcal{K}$ containing $\mathcal{H}$ and a unitary (respectively unitary)  $U \in \mathcal{B}(\mathcal{K})$ such that
\begin{equation} \label{Equation: UnitaryDilation}
A_{n}  = P_{\mathcal{H}} U^{n}|_{\mathcal{H}}, \; \text{for all }\;
n \geq 0.
\end{equation}

Obviously, every unitary dilation is an isometric dilation. Sz.-Nagy
dilation theorem implies, in particular, that every isometry has a
unitary dilation. Therefore, if a sequence admits an
isometric dilation, it also admits a unitary dilation. Consequently,
an operator sequence admits an isometric dilation if and only if it
admits a unitary dilation. However, there is a difference in the
notion of minimality. With notation as above, a unitary dilation $U$
acting on $\mathcal {K}$ is minimal if
$$\mathcal{K}=\overline{\mbox{span}}\{ U^nh: n\in \mathbb{Z}, h\in
\mathcal{H}\}.$$ On the other hand, an isometric dilation $U$ acting
on $\mathcal{K}$ is minimal if
$$\mathcal{K}=\overline{\mbox{span}}\{ U^nh: n\in \mathbb{Z}_+, h\in
\mathcal{H}\}.$$
Let $\mathbb{D}:= \big\{ z\in \mathbb{C}:\; |z| < 1\big\}$ be an open unit disk. The classical  Szeg\H{o} kernel and Poisson kernel  \cite{AR07, LG08} are denoted and defined respectively by
\begin{align*}
S(w,z)&= \frac{1}{ 1-\overline{w}z} {~for~ all~} z,w\in \mathbb{D}. \\
P_r(\theta)&=  \frac{1-r^2}{1-2r\cos{\theta}+ r^2} {~for~ all~} 0\leq r<1 , 0\leq \theta \leq 2\pi.
  \end{align*}
 A Poisson kernel can be written using the Szeg\H o kernel as
 \begin{align*}
                                  P_{r_{1}r_{2}}(\theta-t)
                                 &=S(w,z)+ \overline{S(w, z)} -1 \text{ for all } z,w\in \mathbb{D} \text{ and } z= r_{1} e^{i\theta}, w= r_{2} e^{i t}.
\end{align*}
 Moreover,  $P_{r}(\theta)\geq 0$ for all $0\leq r<1$ and $0\leq \theta \leq 2 \pi.$

Keeping Szeg\H{o}  and Poisson kernels in mind, we  define an
operator valued kernel function for a sequence of operators. Let
${\bf A}= \{A_{n}\}_{n\geq 0}$ be a sequence of contractions in
$\mathcal{B(H)}$. The associated Szeg\H{o} kernel function $S_{\bf
A}: \mathbb{D}\to \mathcal{B(H)}$ is defined as
 \begin{align}\label{Operator valued Szego Kernel}
 S_{\bf A}(z)=\sum_{n=0}^{\infty}z^{n}A^*_{n} \text{~for~all~} z\in  \mathbb{D}.
 \end{align}
We define the associated Poisson kernel function by
\begin{align}\label{Equation:Poisson Kernel}
  P_{\bf A}(z)= S_{\bf A}(z)+ S_{\bf A}(z)^*-I \text{~for all~} z\in \mathbb{D}.
\end{align}
F. H. Vasilescu  introduced  operator-valued Poisson kernel
functions \cite{Vasilescu 1992} for $d$-tuples of operators using
defect operators to study Holomorphic functional calculus. Such
Poisson kernel functions \textcolor{black}{cannot} be extended to operator-valued
sequences as we do not have the semigroup property (i.e.
$A_{n}A_{m}=A_{n+m}$ may not hold). Next, we prove the necessary and sufficient criteria for isometric and unitary dilations in terms of the Poisson kernel.

\begin{thm}\label{Theorem: UnitaryDilation}
Let $\{A_{n}\}_{n \geq 0}$ be a sequence in
$\mathcal{B}(\mathcal{H})$ with $A_{0}=I$.
 Then the following are equivalent:

 (i)  $\{A_{n}\}_{n\geq 0}$ admits a unitary/isometric dilation.

 (ii) For every $c_{0}, c_{1}, \cdots, c_{n} \in \mathbb{C}$ and $n \in \mathbb{N}$ we have
 \begin{equation} \label{unitarykernel}
 \sum\limits_{\ell,k=0}^n\bar{c_{\ell}}c_kA_{k-\ell }\geq 0,
 \end{equation}
  where
 \[
 A_{(k-\ell)} = \left\{\begin{array}{cc}
 A_{k-\ell}, & \mbox k \geq \ell\\
 & \\
 A^{\ast}_{\ell-k}, &   k < \ell.
 \end{array}\right.
 \]

(iii) $P_{\bf A}(z) \geq 0$ in $\mathcal{{B(H)}}$ for all $z\in
\mathbb{D}.$
\end{thm}
\begin{proof}
Before coming to the proof of this result, we point out that  there
is a subtle point here. The condition in Equation
\eqref{unitarykernel} is a priori weaker than having complete
positivity of the $B({\mathcal H})$ valued kernel $K: {\mathbb
Z}\times {\mathbb Z}\to B({\mathcal H})$ defined by
\begin{equation} \label{unitaryCPkernel}
K(\ell, k)= A_{(k-\ell)}, ~~\text{for every}\; \ell, k \in {\mathbb Z}
\end{equation}
as that would mean,
 \begin{equation*}
 \sum _{\ell ,k=0}^n\langle
g_{\ell}, A_{(k-\ell)}g_{k} \rangle \geq 0, \; \text{for all}\; g_0, g_1,
\ldots , g_n\in {\mathcal H}, n\in {\mathbb N}.
\end{equation*}
Proof of $(i) \implies (ii):$ Assume that Equation(\ref{Equation: UnitaryDilation}) holds true. Let $g \in \mathcal{H}$, then for any $N \geq 0$ we get
\begin{align*}
    \Big\langle g, \; \sum\limits_{\ell, k = 0}^{N} \overline{c_{\ell}} c_{k} A_{(k-\ell)}g \Big\rangle &= \Big\langle \begin{bmatrix} c_{0}g\\c_{1}g\\\vdots\\c_{n}g \end{bmatrix}, \; \begin{bmatrix} I & U & \cdots & U^{N}\\U^{\ast} & I&\cdots & U^{N-1}\\ \vdots & \vdots & \ddots & \vdots\\ {U^{\ast}}^{N}& {U^{\ast}}^{N-1} &\cdots & I \end{bmatrix}\begin{bmatrix} c_{0}g\\c_{1}g\\\vdots\\c_{n}g \end{bmatrix}\Big\rangle\\
    & = \Big\langle \sum\limits_{\ell=0}^{N}c_{\ell}U^{\ell}g,\; \sum\limits_{\ell=0}^{N}c_{\ell}U^{\ell}g \Big\rangle\\&\geq 0.
\end{align*}
Proof of $(ii) \implies (i):$ Similar to the proof of \cite[Theroem 2.6]{Paulsen 2002}, let us define the operator system $
\mathcal{S} = \Big\{\sum\limits_{k=-N}^{N} c_{k} e^{ik\theta}:\;
N\geq 0\Big\} \subset C(\mathbb{T})$ and define $\Phi \colon
\mathcal{S} \to \mathcal{B}(\mathcal{H})$ by $\Phi(e^{in\theta}) =
A_{n}$ for every $n \in \mathbb{Z}$. It is enough to show that $\Phi$ maps strictly positive function to a positive operator in  $\mathcal{B}(\mathcal{H})$. Suppose that $f \in
\mathcal{S}$ is strictly positive then by Fej\'{e}r-Riesz theorem
(see \cite[Lemma 2.5]{Paulsen 2002}), we see that $f(e^{i\theta}) =
\sum\limits_{i,j = 0}^{N} \overline{c_{\ell}}c_{k}
e^{i(k-\ell)\theta}$, for some $c_{0}, c_{1}, \hdots, c_{N} \in \mathbb{C}$ and
hence
$$\Phi(f(e^{i\theta})) = \sum\limits_{\ell, k = 0}^{N}
\overline{c_{\ell}}c_{k} \Phi ( e^{i(k-\ell) \theta})=
\sum\limits_{\ell, k = 0}^{N} \overline{c_{\ell}} c_{k} A_{(k-\ell)}
\geq 0.$$
Here $\Phi$ is a positive map and so bounded. Since $\mathcal{S}$ is dense in $C(\mathbb{T})$ by Stone-Weierstrass we see that $\Phi$ can be
extended to a positive map on $C(\mathbb{T})$ by \cite[Exercise
2.2]{Paulsen 2002}. Again we denote it by $\Phi$. Therefore, it
follows from \cite[Theorem 3.11]{Paulsen 2002} that $\Phi$ is
completely positive map on $C(\mathbb{T})$.
Now by Stinespring's  theorem there exists a triple  $({\mathcal K},
\pi , V)$ where ${\mathcal K}$ is a Hilbert space, $\pi : C({\mathbb
T})\to B({\mathcal K})$ is a unital $*$-homomorphism, $V\colon
{\mathcal H}\to {\mathcal K}$ a bounded linear map satisfying
\begin{equation}\label{Stinespring1} \Phi (f) = V^*\pi (f)V, ~~\forall f\in C({\mathbb
T}).\end{equation}

Since $\Phi $ is unital, $V$ is an isometry and we may consider
${\mathcal H}$ as a subspace of ${\mathcal K}$ by identifying $h\in
{\mathcal H}$ with $Vh$ in ${\mathcal K}$. Then \textcolor{black}{equation
\eqref{Stinespring1}} reads as
$$\Phi (f) = P_{\mathcal H}\pi (f)|_{\mathcal H}, f\in C({\mathbb
T}).$$ If $f_0(z)=z$, then $\pi(f_{0})$ is unitary operator as,
\begin{equation*}
\pi (f_0)^{\ast} \pi(f_{0}) = \pi(f_{0}) \pi(f_{0})^{\ast} = \pi(|f_{0}(z)|^{2}) = I.
\end{equation*}
By taking $U = \pi(f_{0})$ in ${\mathcal B(K)}$ we get
\begin{align*}
P_{\mathcal{H}} U^{n}\big|_{\mathcal{H}} = P_{\mathcal{H}}\pi(f_{0})^{n}\big|_{\mathcal{H}} = P_{\mathcal{H}}\pi(f^{n}_{0})\big|_{\mathcal{H}} = \Phi(z^{n}) = A_{n}, \; \text{for}\; n \geq 0.
\end{align*}
\noindent Proof of $(i)\implies (iii):$ Suppose $\{A_{n}: n\geq 0\}$ admits a unitary dilation. Then there is a unitary operator $U\in \mathcal{B(K)}$ for some Hilbert space $\mathcal{K}\supseteq \mathcal{H}$ such that
 \begin{align}\label{Dilation Relation}
    A_{n}= &P_{\mathcal{H}}U^n|_{\mathcal{H}} \text{~for all~} n \geq 0.
 \end{align}
 First notice that for each $z\in \mathbb{D},$ $(1-z U^*)$ is invertible and $(1-zU^{*})^{-1}= 1+ zU^*+ z^{2}U^{* 2}+ \cdots$ and
 $$\Vert I+ zU^*+ z^{2}U^{* 2}+ \cdots\Vert \leq 1+ \vert z\vert + \vert z\vert^2 +\cdots =\frac{1}{1-\vert z\vert}.$$
 Now a simple computation ensures that $$(I-zU^{*})^{-1}+ (I-\overline{z}U)^{-1}-I= (1-\vert z\vert^2)(I-\overline{z}U)(I-\overline{z}U)^*.$$
 Also, it is immediate to see that $(I-zU^{*})^{-1}+ (I-\overline{z}U)^{-1}-I \geq 0$ as $\vert z\vert<1.$ It implies that
  $ P_{\mathcal{H}}\Big( (I-zU^{*})^{-1}+ (I-\overline{z}U)^{-1}-I\Big)|_{\mathcal{H}} \geq 0.$  Using the hypothesis of dilation Equation \eqref{Dilation Relation}, we see that
  \begin{align*}
 \Vert I \Vert +\Vert zA^{*}_{1} \Vert+ \Vert z^{2}A^{*}_{2} \Vert+& \cdots + \Vert z^{n}A_{n}^{*}  \Vert\\
          & \leq      1 + \vert z\vert \Vert A^{*}_{1} \Vert+ \vert z^{2} \vert\Vert A^{*}_{2}\Vert +\cdots + \vert z^{n}\vert \Vert A_{n}^{*}\Vert \\
   &\leq   1 + \vert z\vert \Vert U^{*} \Vert+ \vert z^{2} \vert\Vert U^{*2}\Vert +\cdots + \vert z^{n}\vert \Vert U^{*n}\Vert \\
   &=1 + \vert z\vert + \vert z \vert^{2}  +\cdots + \vert z \vert^{n} \\
   &\leq  \frac{1}{1- \vert z\vert} \text{~for all~} z\in \mathbb{D}.
  \end{align*}
  Thus the series $\sum\limits_{n=0}^{\infty}z^{n}A^*_{n}$  converges in norm absolutely for all $z\in \mathbb{D}.$ Therefore, $S_{\bf A}(z)$  is well defined. Further, we compute that
\begin{align*}
S_{\bf A}(z)+ S_{\bf A}(z)^{*}-I = &\sum_{n=0}^{\infty}z^{n}A^*_{n}+ \sum_{m=0}^{\infty}\overline{z}^{m}A_{m}-I\\
=& \sum_{n=0}^{\infty}z^{n} P_{\mathcal{H}}U^{*n}|_{\mathcal{H}}+ \sum_{m=0}^{\infty}\overline{z}^{m} P_{\mathcal{H}}U^{m}|_{\mathcal{H}} -I\\
                                   = & P_{\mathcal{H}}\big (\sum_{n=0}^{\infty}z^{n} U^{*n}\big)|_{\mathcal{H}}+ P_{\mathcal{H}}\big (\sum_{m=0}^{\infty}\overline{z}^{m} U^{m}\big)|_{\mathcal{H}} -I\\
                                   =&P_{\mathcal{H}}\Big( (I-zU^{*})^{-1}+ (I-\overline{z}U)^{-1}-I\Big)|_{\mathcal{H}}.
                                    \end{align*}
Since $P_{\mathcal{H}}\Big( (I-zU^{*})^{-1}+
(I-\overline{z}U)^{-1}-I\Big)|_{\mathcal{H}}\geq 0,$ therefore we
have $P_{\bf A}(z)\geq 0$ for all $z\in \mathbb{D}.$\\

\noindent Proof of $(iii) \implies (i):$ We prove this result motivated by Theorem 2.14 of \cite{Paulsen 2002}. Assume that  $P_{\bf A}(z) \geq 0$ for all $z\in \mathbb{D}.$  Let $0\leq r <1,$ then for
each $l\in \mathbb{Z},$ we have
\begin{align}\label{Comptaion of using Poisson Kernel}
  \notag  \frac{1}{2\pi}\int_{0}^{2\pi} e^{i l \theta}P_{\bf A}(re^{i\theta})d\theta 
\notag    &= \frac{1}{2\pi}\int_{0}^{2\pi} e^{i l \theta}\big( \sum_{n=0}^{\infty}r^{n} e^{in\theta}A^*_{n} +\sum_{m=0}^{\infty}r^{m} e^{-im\theta}A_{m} \big)d\theta \\
 &=\left\{ \begin{array}{cc}
     r^{l}A_{l} &\text{if} \; l \geq 0 \\
   &\\
   r^{-l}A_{-l}^{*} &\text{if} \; l<0. \\
\end{array} \right.
\end{align}
Let us define $\varphi \colon \mathcal{S} \to \mathcal{B}(\mathcal{H})$ by
$$
\varphi\Big(\sum_{n=0}^{N} p_{n}e^{in\theta}+ \sum_{m=0}^{N}
\overline{q_{m}}e^{-im\theta}\Big)=\sum_{n=0}^{N}
p_{n}A_{n}+\sum_{m=0}^{N} \overline{q_{m}} A_{m}^{*},
$$
for all $\sum\limits_{n=0}^{N} p_{n}e^{in\theta}+ \sum\limits_{m=0}^{N}
\overline{q_{m}}e^{-im\theta} \in \mathcal{S}.$ Now it is enough to
show that the map $\varphi: \mathcal{S}\to \mathcal{B}(\mathcal{H})$
is a positive map. Now, we observe the following
\begin{align*}
   \frac{1}{2\pi}\int_{0}^{2\pi}\Big( \sum_{n=0}^{N} p_{n} e^{i n\theta} \Big) P_{\bf A}(re^{i\theta}) d\theta
   &= \sum_{n=0}^{N} p_{n} \frac{1}{2\pi}\int_{0}^{2\pi} e^{i n\theta}  P_{\bf A}(re^{i\theta}) d\theta\\
   &=\sum_{n=0}^{N}p_{n} r^n A_{n}.
\end{align*}
Similarly, we have
\begin{align*}
   \frac{1}{2\pi}\int_{0}^{2\pi}\Big( \sum_{m=0}^{N} \overline{q_{m}}e^{-im\theta} \Big) P_{\bf A}(re^{i\theta}) d\theta
   &=\sum_{m=0}^{N} \overline{q_{m}} r^m A_{m}^{*}.
   \end{align*}
    Combining all, we finally obtain that
\begin{align} \label{Equation: PoissonIntegral}
\notag \varphi\big(\sum_{n=0}^{N} p_{n}e^{in\theta}&+ \sum_{m=0}^{N} \overline{q_{m}}e^{-im\theta}\big)\\
&=\sum_{n=0}^{N}p_{n}A_{n}+\sum_{m=0}^{N} \overline{q_{m}} A_{m}^{*} \\
 \notag &= \lim_{r \to 1-} \big(\sum_{n=0}^{N}p_{n}r^{n}A_{n}+\sum_{m=0}^{N} \overline{q_{m}}r^{m} A_{m}^{*}\big)\\
&=\lim_{r \to 1-}\Big( \frac{1}{2\pi}\int_{0}^{2\pi}\big(
\sum_{n=0}^{N} p_{n} e^{i n\theta} + \sum_{m=0}^{N}
\overline{q_{m}}e^{-im\theta}\big) P_{\bf A}(re^{i\theta}) d\theta
\Big).
   \end{align}
 Notice that for all $0\leq \theta \leq 2\pi$ and $0\leq r<1,$ we have $P_{\bf A}(re^{i\theta})\geq 0.$ This means that  whenever
  $ \sum\limits_{n=0}^{N} p_{n} e^{i n\theta} +
\sum\limits_{m=0}^{N} \overline{q_{m}}e^{-im\theta} \geq 0$ we have
 \begin{equation*}
 \int_{0}^{2\pi}\Big( \sum_{n=0}^{N} p_{n} e^{i n\theta} + \sum_{m=0}^{N} \overline{q_{m}}e^{-im\theta}\Big) P_{\bf A}(re^{i\theta}) \; d\theta \geq 0,\; \text{for evey}\; 0 \leq r <1.
\end{equation*}
Thus $\varphi$ is positive by Equation \eqref{Equation: PoissonIntegral}. This completes the proof. 
\end{proof}
\noindent We understand Theorem \ref{Theorem: UnitaryDilation} in more detail in the next section.
\subsection{Concrete Unitary dilation}

 Let $V$ be a minimal isometric dilation on some Hilbert space ${\mathcal K}$ for a sequence
 of contractions $\{A_n\}_{n\geq 0}$ on $\mathcal{H}.$  In particular, ${\mathcal K}= \overline{\rm span}\{V^{n}\mathcal{H}: n\geq
0\}$. Then by Lemma \ref{Upper Hessenberg}, ${\mathcal K}$ decomposes
as
 $${\mathcal K}={\mathcal H}_0\oplus {\mathcal H}_1\oplus \cdots ,$$
 where ${\mathcal H}_0={\mathcal H}$ and with respect to this decomposition the operator $V$ has
 the block form:
\begin{equation}\label{Equation: Schaffer}
    V= \begin{bmatrix}
    V_{00} & V_{01} & V_{02} & \cdots & V_{0n} &
    \cdots \\
    V_{10} & V_{11} & V_{12} & \cdots & V_{1n} & \cdots \\
    0      & V_{21} & V_{22} & \cdots & V_{2n} & \cdots \\
    0      & 0      & V_{32} & \cdots & V_{3n} & \cdots \\
    \vdots & \ddots & \ddots & \vdots & \vdots & \vdots \\
    \end{bmatrix}.
\end{equation}

We wish to construct $V_{ij}$'s using the given sequence $\{A_n\}.$
Recall that  ${\mathcal H}_{n]}=\overline{\mbox span} \{ V^mh: 0\leq
m\leq n-1, h\in {\mathcal H}\}$ and ${\mathcal H}_n= {\mathcal
H}_{n]}\bigcap {{\mathcal H}_{(n-1)]}^{\perp }}$ for $n\geq 1$ with
${\mathcal H}_{0]}= {\mathcal H}_0.$

 We try to determine the
blocks using the sequence $\{A_n\}_{n\geq 0}$ and the dilation
property. Like in the case of self-adjoint dilation we assume
invertibility of concerned operators as and when required. Since $V$
is an isometric dilation for the moment sequence $\{A_{n}\}_{n \geq
0}$ we have $P_{\mathcal{H}}V^{n}\big|_{\mathcal{H}} = A_{n}$. In
other words, $(V^{n})_{00} = A_{n},\; \text{for}\; n \geq 0.$ Let us
compute the first column of $V$. Clearly, $V_{00} = A_{1}$ and since
$V$ is an isometry,  we get
\begin{equation*}
    I = (V^{\ast}V)_{00} = V_{00}^{\ast}V_{00}+V_{10}^{\ast}V_{10} = A_{1}^{\ast}A_{1}+ V_{10}^{\ast}V_{10}.
\end{equation*}
One can choose that $V_{10} = (I-A_{1}^{\ast}A_{1})^{1/2}$ as
$A_{1}$ is a contraction. The first row of $V$ is given by the
following recursive relation: For every $n\in \mathbb{N},$ we have
\begin{align*}
    A_{n} &= (V^{n})_{00}\\
          &= \sum\limits_{0\leq r_{1}, r_{2},\hdots,r_{(n-1)}\leq n-1} V_{0r_{1}} V_{r_{1}r_{2}}\cdots V_{r_{(n-1)}0} \\
          &= \sum\limits_{0\leq r_{1}, r_{2},\cdots, r_{n-1} \leq n-1\; \&\; r_{1} \neq n-1} \!\!\! \! \! \! V_{0r_{1}} V_{r_{1}r_{2}}\cdots V_{r_{(n-1)}0} \\
          & ~~ \; \;\;\; + \sum\limits_{0\leq r_{2},\cdots, r_{n-1} \leq n-1} V_{0(n-1)} V_{(n-1) r_{2}}\cdots V_{r_{(n-1)}0}.
\end{align*}
Since $V_{ij} = 0$ whenever $i>j+1$, it implies that
\begin{align*}
    \sum\limits_{0\leq r_{2},\cdots, r_{n-1} \leq n-1} V_{0(n-1)} &V_{(n-1) r_{2}}\cdots V_{r_{(n-1)}0}\\
     &= V_{0(n-1)} V_{(n-1)(n-2)} V_{(n-2)(n-3)}\cdots V_{10}.
\end{align*}
Therefore,
\begin{align*}
    V_{0n} &=  \Big[ A_{n} - \sum\limits_{0\leq r_{1}, r_{2},\cdots, r_{n-1} \leq n-1\; \&\; r_{1} \neq n-1}\!\!\!\!\!\!  V_{0r_{1}} V_{r_{1}r_{2}} \cdots V_{r_{(n-1)}0}\Big] \\
    & ~~~~~~~~~~~\;\;\;\Big(V_{(n-1)(n-2)}V_{(n-2)(n-3)}\cdots V_{10}\Big)^{-1}.
\end{align*}
Next, we use the isometric property of $V$ to obtain expression for
the $n^{\text{th}}$ column of $V$. For a fixed $n\in \mathbb{N}$,
suppose the first $(n-1)$ columns are known then $V_{1n}$ is given
by considering the inner product of the first column with the
$n^{\text{th}}$ column which is zero. That is, $ V_{00}^{\ast}V_{0n} + V_{10}^{\ast}V_{1n} = 0,$ and therefore
\begin{align*}
    V_{1n} &= - (I-A_{1}^{\ast}A_{1})^{-1/2} A^{\ast}_{1}V_{0n}.
\end{align*}
Indeed $V_{kn}$ for $0< k \leq n$ is given by the inner product of
the $(k-1)^{\text{th}}$ column with the $n^{\text{th}}$ column which
is again zero. That is,
 $ \sum\limits_{i=0}^{k} V^{\ast}_{i(k-1)}V_{in} = 0.$
Equivalently, $$V^{\ast}_{k(k-1)}V_{kn}+ \sum\limits_{i=0}^{k-1}
V^{\ast}_{i(k-1)}V_{in}=0.$$ This implies that
\begin{equation*}
    V_{kn}=- \big(V_{k(k-1)}^{\ast}\big)^{-1}\sum\limits_{i=0}^{k-1}V_{i(k-1)}^{\ast}V_{in}.
\end{equation*}
For $k=n+1$, we have $ \sum\limits_{i=0}^{n+1}V_{in}^{\ast}V_{in} =I$ and equivalently,   $V_{(n+1)n}^{\ast}V_{(n+1)n} + \sum\limits_{i=0}^{n}V_{in}^{\ast}V_{in} = I.$ It implies that
\begin{align*}
    V_{(n+1)n}^{\ast}V_{(n+1)n} &= I-\sum\limits_{i=0}^{n}V_{in}^{\ast}V_{in}.
\end{align*}
One can choose that $$V_{(n+1)n} =
\big[I-\sum\limits_{i=0}^{n}V_{in}^{\ast}V_{in}\big]^{1/2}.$$
 As a
result, the recurrence relations are given by 
$$V_{00} = A_{1},\; V_{10} = (I-A_{1}^{\ast}A_{1})^{1/2},$$
\begin{align*}
    V_{0n} &=  \Big[ A_{n} - \sum\limits_{0\leq r_{1}, r_{2},\cdots, r_{n-1} \leq n-1\; \&\; r_{1} \neq n-1}\!\!\!\!\!\!  V_{0r_{1}} V_{r_{1}r_{2}} \cdots V_{r_{(n-1)}0}\Big] \\
    & ~~~~~~~~~~~\;\;\;\Big(V_{(n-1)(n-2)}V_{(n-2)(n-3)}\cdots V_{10}\Big)^{-1}.
\end{align*}  and
\[V_{kn} = \left\{ \begin{array}{cc}
- \big(V_{k(k-1)}^{\ast}\big)^{-1}\sum\limits_{i=0}^{k-1}V_{i(k-1)}^{\ast}V_{in} & \mbox{ if\; $0<k\leq n$} \\
&\\
\Big[ I - \sum\limits_{i=0}^{n} V_{in}^{\ast}V_{in}\Big]^{1/2} &
\mbox{if \; $k = n+1$}
\end{array}\right.\]
for $n>0.$
\\

\noindent The  above computations lead to the following result.
\begin{thm}\label{Theorem: isometry}
Let $\{A_{n}\}_{n \geq 0}$ be a sequence of contractions on
$\mathcal{H}$ and $A_{0} = I$ admitting an isometric dilation. If
the inverses appearing in the recurrence relations above are well-defined bounded operators, then these formulae provide a minimal
isometric dilation with the blocks of $V$ described above.
\end{thm}

Now we look at unitary dilations.

\begin{thm}\label{Theorem: minimal unitary} Suppose $U$ is a unitary on a Hilbert space $\mathcal{K}$ and  $\mathcal{H}$ is a closed subspace
of $\mathcal{K}$, satisfying $\mathcal{K}=\overline{\rm span}\{
U^nh: h\in \mathcal{H}, n\in \mathbb{Z}\}.$
 Then ${\mathcal K}$ decomposes as
 $${\mathcal K}= \cdots \oplus {\mathcal H}_{-1}\oplus {\mathcal H}_{0}\oplus {\mathcal H}_1\oplus \cdots ,$$
 where ${\mathcal H}_0={\mathcal H}$ so that with respect to this decomposition  $U$ has the form:
\[
U= \left[\begin{array}{@{}ccccc|ccccccc@{}}
   &  \ddots    &  \vdots & \vdots                       & \vdots          &  \vdots                  &   \vdots                  &   \vdots                                 &                       \\
   &    &I   &0                      & 0                  & 0                           &0                              &0                                             &                  \\
   &  \hdots             & 0 & I                        & 0                    &0                               &   0                              & 0                                              &\hdots            \\
      \hline
   & \hdots        & 0   &0                      &U_{0(-1)}         & U_{00}                &U_{01}                 &U_{02}                                    &\hdots        \\
   &  \hdots      &  0 &0                      &U_{1(-1)}         & U_{10}                &  U_{11}                  &  U_{12}                                  &  \hdots     \\
   & \hdots       &  0 & 0                     & U_{2(-1)}        & 0                         &U_{21}                   &U_{22}                                     &\hdots                         \\
   &  \hdots      &  0 &0                      & U_{3(-1)}        & 0                         &0                            &U_{32}                                     &  \hdots                       \\
   &                  & \vdots  &\vdots               & \vdots           &  \vdots                & \vdots                    &  \ddots                                    &\ddots                       \\
      \end{array}\right].
\]
\end{thm}

\begin{proof} Take ${\mathcal H}_{0]}= {\mathcal H}_0$ and ${\mathcal H}_{n]}=\overline{\mbox{span}} \{ U^mh: 0\leq m\leq
n-1, h\in {\mathcal H}\}$ and ${\mathcal H}_n= {\mathcal
H}_{n]}\bigcap {{\mathcal H}_{(n-1)]}^{\perp }}$ for $n\geq 1.$
 $\mathcal{K}_{+}=\bigvee^{\infty}\limits_{n=0} U^{n} \mathcal{H},$ and $\mathcal{K}_{-}=\mathcal{K}^{\perp}_{+}.$ Then clearly ${\mathcal K}_{+}=
\oplus_{n\geq 0} {\mathcal H}_n$
 and  $U ({\mathcal
H}_{n]})\subseteq {\mathcal H}_{(n+1)]}.$ It implies that $n$-th column of the block matrix satisfies the desired property for all $n\geq 0.$
Notice that $U^*(\mathcal{K}_{-})\subseteq \mathcal{K}_{-}$ as $\mathcal{K}_{-}= \mathcal{K}^{\perp}_{+}$ and $U(\mathcal{K}_{+})\subseteq \mathcal{K}_{+}.$ We claim that
$$
\bigcap_{n=0}^{\infty}U^{*n}(\mathcal{K}_{-}) =\{0\}.
$$
Let  $x\in \bigcap\limits_{n=0}^{\infty}U^{*n}(\mathcal{K}_{-})
=\{0\}.$ Let $n\geq 0,$ and $h \in \mathcal{H}$ notice that $x=
U^{*n}x_{n}$ for some $x_{n}\in \mathcal{K}_{-}.$  Then $\langle x,
U^{*n}h \rangle = \langle U^{*n}x_{n}, U^{*n}h\rangle = \langle
x_{n}, h\rangle = 0,$ and $\langle x, U^{m}h \rangle=0$ for all
$m\geq 0. $ It follows that $x=0.$ Hence $U^*: \mathcal{K}_{-}\to
\mathcal{K}_{-}$ is a shift with wandering subspace
$$\mathcal{H}_{-1}:= \mathcal{K}_{-}\ominus U^*\mathcal{K}_{-}.$$
Therefore $\mathcal{K}_{-}$ can be decomposed as
$$
\mathcal{K}_{-1}= \cdots \oplus U^{*2} \mathcal{H}_{-1} \oplus U^{*}
\mathcal{H}_{-1}\oplus \mathcal{H}_{-1}.
$$
Now, it is enough to see that $U(\mathcal{H}_{-1})\subseteq
\mathcal{K}_{+}.$ Consider $x\in \mathcal{H}_{-1},$ then for any
$y\in \mathcal{K}_{-},$ we see that $\langle Ux, y \rangle =\langle
x, U^*y\rangle =0,$ as $ x\perp U^{*}\mathcal{K}_{-}.$ This
completes the proof.
\end{proof}


\section{$\mathcal{C}_{A}$-class operators}

Inspired by  Sz.-Nagy's  dilation theorem for contractions, Berger
and \textcolor{black}{Stampfli} \cite{Ber Stam 67}, obtained the following interesting
theorem. Suppose $T\in \mathcal{B(H)}$. Then the numerical radius
$w(T) \leq 1$ if and only if the operator  sequence $\{
\frac{1}{2}T^{n},\; n \geq 1\}$ admits a unitary dilation.  This led
to the following definition. Fix $\rho >0$. Then a bounded operator
$T\in \mathcal{B(H)}$ is said to be in $\mathcal{C}_{\rho }$-class
if there is a unitary operator $U$ on some Hilbert space
$\mathcal{K} \supset \mathcal{H}$ such that $$\frac{T^{n}}{\rho} =
P_{\mathcal{H}}U^{n}\big|_{\mathcal{H}}$$ for $n \geq 1$ (see page 43
of \cite{SF10}). Thanks to Sz.-Nagy's dilation theorem,
  $\mathcal{C}_{1}$-class is  precisely the set of all contractions in $\mathcal{B(H)}$. Similarly from the
   \textcolor{black}{Berger-Stampfli's} Theorem \cite{Ber Stam 67}
 we know that $\mathcal{C}_{2}$-class is the set of all operators in $\mathcal{B(H)}$ whose numerical radius is less than equal
 to one. Now we recall a theorem from \cite{SF10} that  characterises  bounded linear operators of $\mathcal{C}_{\rho}$-class, where $\rho >0.$

 \begin{thm} \cite[Theorem 11.1]{SF10}
 Let $T \in \mathcal{B}(\mathcal{H})$ and $\rho >0.$ Then $T \in \mathcal{C}_{\rho}$ if and only if
 \begin{equation*}
 \Big( \frac{2}{\rho} - 1 \Big) \|zTh\|^{2} + 2 \Big( 1-\frac{1}{\rho}\Big) Re \big( \langle zTh,\; h\rangle\big)\; \leq \; \|h\|^{2},\; \text{for all}\; h \in \mathcal{H},\; |z| \leq 1.
 \end{equation*}
 Further, it is equivalent to satisfying the following condition:
 \begin{equation}\label{Corollary: equvallent criteria of Crho-class}
 \big( \rho -2 \big) \big\| (I-zT) h \big\|^{2} + 2 Re \big( \langle (I-zT)h,\; h\rangle\big) \geq 0,
 \end{equation}
 for all $h \in \mathcal{H},\; |z| \leq 1.$
 \end{thm}

M. A. Dritschel, H. J. Woerdeman \cite{DW96} have  developed a model
theory for the $\mathcal{C}_{2}$-class.  Over the past few decades,
there has been an extensive study of $\mathcal{C}_{\rho}$-class
operators. To mention some of them, we refer \cite{AT75, BS67, Ber
Stam 67, DMW99,  Hol71, Hok68, NF66} and references therein.
Generalizing this notion in a natural way   H. Langer  introduced
the $\mathcal{C}_{A}$-class. Let $A\in \mathcal{B(H)}$ be a positive
invertible operator. Then an operator $T\in \mathcal{B(H)}$ is said
to be of the $\mathcal{C}_{A}$-class if the  operator-valued sequence
$\{A_{n}\}_{n \geq 1}$, where
   \begin{align}\label{Equation: Toeplitz moment sequence}
   A_{n}:= A^{-\frac{1}{2}} T^{n}A^{-\frac{1}{2}},\text{ for } n\geq 1,
   \end{align}
   admits unitary dilation. The initial reference we could find for this notion is  \cite{SF10}.

\begin{thm}\label{Thm: CA class equivalent}
Let $A\in \mathcal{B(H)}$ be a positive invertible operator and the
operator $T\in \mathcal{B(H)}.$ Then the following are equivalent:
\begin{itemize}
\item [(i)] $T\in \mathcal{C}_{A}$-class.

\item[(ii)]For every $N\in \mathbb{N}$ and $c_{1}, \cdots , c_{N}\in \mathbb{C}$, we have
  \begin{equation}\label{Equation: ZetaA}
    \sum_{\ell,k=0}^{N}\overline{c_{\ell}}c_{k}\zeta_{A}(k-\ell) \geq 0 \text{~for each~} N\in \mathbb{N},
 \end{equation}
where $\zeta_{A}(n)$ is defined as follows:
\begin{align*}
\zeta_{A}(n):=\left\{ \begin{array}{ll} A^{-\frac{1}{2}}T^nA^{-\frac{1}{2}} & {\mbox if} ~n > 0;\\
                                   I &\mbox{if}  ~n =0;\\
                                   A^{-\frac{1}{2}}T^{*-n}A^{-\frac{1}{2}} & \mbox{if}~ n < 0.\\
                                   \end{array}\right.
\end{align*}
\item[(iii)] $T$ satisfies the following:
$$(I - zT)^{-1}+  (I-\overline{z}T^*)^{-1}+ A-2I \geq 0 \text{~~~for all~} \vert z \vert <1.$$
\item [(iv)]$T$ satisfies the following:
\begin{align}\label{Equation: Ca class}
\left\langle(I-zT)h, (A-2I)(I-zT)h \right\rangle+ 2{\rm Re} \langle
h, (I-zT)h \rangle \geq 0 \text{~~~for all~} \vert z \vert <1, h\in
\mathcal{H}.
\end{align}
\end{itemize}
     \end{thm}
\begin{proof}
In view of Theorem \ref{Theorem: UnitaryDilation}, it is immediate
to see that (i) and (ii) are equivalent. Now, \textcolor{black}{ we prove other implications below}.

  (i)$\implies$ (iii): Let ${\bf A}=  \{A_{n}\}_{n\geq0 }$,  where $A_{n}= A^{-\frac{1}{2}}T^nA^{-\frac{1}{2}}$ for $n\geq1$ and $A_{0}=I.$ Assume that  $T\in \mathcal{C}_{A}.$ Then by Theorem \ref{Theorem: UnitaryDilation}, $P_{\bf A}(\overline{z}) \geq0$ for all $z\in\mathbb{D}.$ In other words,
  \begin{align*}
  (I + zA^{-\frac{1}{2}}TA^{-\frac{1}{2}}+ &z^2A^{-\frac{1}{2}}T^2A^{-\frac{1}{2}}+ \hdots  \infty)\\
  &+  (I+\overline{z}A^{-\frac{1}{2}}T^*A^{-\frac{1}{2}}+ \overline{z}^2A^{-\frac{1}{2}}T^{*2}A^{-\frac{1}{2}}+ \hdots  \infty) -I \geq 0.
  \end{align*}
It follow that $$(I - zT)^{-1}+  (I-\overline{z}T^*)^{-1}+ A-2I \geq 0.$$
 (iii)$\implies$ (iv):
  Let $h\in \mathcal{H},$ then for all $z\in \mathbb{D},$ we have
  \begin{align*}
  \Big\langle   (I-zT)h, \big( (I - zT)^{-1}+ &(I-\overline{z}T^*)^{-1}\big)(I-zT)h \Big\rangle \\
  &+  \Big \langle (I - zT), (A-2I)(I-zT)h \Big\rangle \geq 0.
  \end{align*}
  Notice that  $$ \left\langle  (I-zT)h, \big( (I - zT)^{-1}+  (I-\overline{z}T^*)^{-1}\big)(I-zT)h \right\rangle= 2 Re\langle (I-zT)h,  h\rangle.$$ It implies that $$\left\langle(I-zT)h, (A-2I)(I-zT)h \right\rangle+ 2{\rm Re} \langle
(I-zT)h, h \rangle \geq 0 \text{~~~for all~} \vert z \vert <1, h\in
\mathcal{H}.$$

 \noindent (iv)$\implies$ (i): Assume that $$\left\langle(I-zT)h, (A-2I)(I-zT)h \right\rangle+ 2{\rm Re} \langle
(I-zT)h, h \rangle \geq 0 \text{~~~for all~} \vert z \vert <1, h\in
\mathcal{H}.$$ Then by reverse computation, we can show that $$P_{\bf
A}(\overline{z})\geq 0$$ for all $z\in \mathbb{D}.$ Hence by Theorem
\ref{Theorem: UnitaryDilation},  $T\in \mathcal{C}_{A}.$
 \end{proof}
One can see that Theorem 2.15 of \cite{Suen 98} obtains the equivalence of $(i)$ and $(iv)$ with the assumption that
 $A$ satisfies $ mI\leq A \leq MI$, for some $m, M >0$. However, our proof based on the idea of a Poisson kernel is different from the one given in \cite{Suen 98}. Furthermore, Theorem \ref{Thm: CA class equivalent} guarantees that the spectrum of $T$ is within the closed unit disc.
\begin{prop} Let $T\in \mathcal{C}_{A}.$ Then $$\sigma(T)\subseteq \overline{\mathbb{D}},$$  where $\sigma(T)$ denote the spectrum of $T$ and $\mathbb{D}$ denote the open unit disc of $\mathbb{C}.$
\end{prop}

\begin{proof} Suppose $T\in \mathcal{C}_{A}$ then $T$ satisfies  Equation \eqref{Equation: Ca class} from the Theorem \ref{Thm: CA class equivalent}.
 Assume that \textcolor{black}{$\sigma(T) \not\subset \overline{\mathbb{D}}$}. Then there exists a $z_{0}\in \mathbb{C}$
such that $|z_{0}|<1$ and $\frac{1}{z_{0}} \in \sigma_{app}(T)$ the
approximate point spectrum of $T$. So there exists a sequence
$\{h_{n}\}$ in $\mathcal{H}$ with $\Vert h_{n}\Vert=1$ such that
$$\lim\limits_{n\to \infty}(\frac{1}{z_{0}}I-T)h_{n}=0.$$ It
immediately follows that $$\lim\limits_{n\to
\infty}(I-z_{0}T)h_{n}=0, \lim\limits_{n\to \infty} \Vert
z_{0}Th_{n} \Vert=1, \text{ and }\lim\limits_{n\to \infty}\langle
z_{0}Th_{n},h_{n}\rangle=1.$$

Let us choose $r>0$ sufficiently small such that $r \vert z_{0}
\vert+ \vert z_{0} \vert<1.$ Then for all $h_{n}$, we have from
assumption that
\begin{align*}
0\leq &\big\langle(I-z_{0}T-rz_{0}T)h_{n}, (A-2I)(I-z_{0}T-rz_{0}T)h_{n} \big\rangle\\
&~~~~~~~~~~~~ + 2{\rm Re} \big\langle (I-z_{0}T-rz_{0}T)h_{n},h_{n} \big\rangle  \\
\leq &\Vert A-2I \Vert \big\langle(I-z_{0}T-rz_{0}T)h_{n}, (I-z_{0}T-rz_{0}T)h_{n} \big\rangle\\
&~~~~~~~~~~~~  + 2{\rm Re} \big\langle (I-z_{0}T-rz_{0}T)h_{n},h_{n} \big\rangle \\
\leq &\Vert A-2I \Vert\Vert (I-z_{0}T-rz_{0}T)h_{n}\Vert^2 + 2{\rm Re} \big\langle (I-z_{0}T-rz_{0}T)h_{n},h_{n} \big\rangle \\
\leq &\Vert A-2I \Vert\Vert (I-z_{0}T)h_{n}-rz_{0}Th_{n}\Vert^2 \\
&~~~~~~~~~~~~ + 2{\rm Re} \big\langle (I-z_{0}T)h_{n},h_{n} \rangle -2r {\rm Re} \big\langle
z_{0}Th_{n},h_{n} \big\rangle
    \end{align*}
    Taking limit over $n$ tends to infinity, we have $$0\leq \Vert A-2I\Vert r^2-2r.$$ It implies that $$2 \leq \Vert A-2I\Vert r$$ for sufficiently small $r>0.$
    Which is a contradiction, thus $\sigma(T)\subseteq \overline{\mathbb{D}}.$
\end{proof}

In the next result, employing standard techniques, we provide another
necessary and sufficient criterion in terms of positive maps. 
\begin{thm}\label{Theorem: C_A class and CP map}
Let $A, T \in \mathcal{B(H)}$ where $A$ is a positive invertible operator and let $\mathcal{S}\subseteq C(\mathbb{T})$ be the operator system defined by
\begin{equation*}
\mathcal{S} = \big\{ p(e^{i\theta}) + \overline{q(e^{i\theta})}:\; p, q \; \text{are polynomials}\big\}.
\end{equation*}
Then the map $\varphi_{A}:\mathcal{S}\subseteq C(\mathbb{T})\to
\mathcal{B(H)}$ defined by
\begin{align*}
\varphi_{A}\big(p(e^{i\theta})+ \overline{q(e^{i\theta})}\big):=
p(T)+ q(T)^*+ (A-I) \big(p(0)+ \overline{q(0)}\big)I.
\end{align*}
is a positive map if and only if the operator $T$
is in $\mathcal{C}_{A}$-class.
  \end{thm}
\begin{proof}
 Let $T\in {\mathcal{C}}_{A}.$ Now, let
$f(e^{i\theta})=\sum\limits_{n=-N}^N a_{n}e^{i n\theta}\in
\mathcal{S}$ be strictly positive. Then applying Riesz-Fejer
theorem, we see that $f(e^{i\theta})= \sum\limits_{0\leq \ell,k\leq
N} \overline{c_{\ell}}c_{k} e^{i(k-\ell)\theta}$  for some $c_{0},
\hdots, c_{N}\in \mathbb{C}.$ This implies that
\begin{align*}
\varphi_{A}(f(e^{i\theta}))&= \sum\limits_{\ell=0}^{N} \vert c_{\ell}\vert^2+ \sum\limits_{0\leq \ell\neq k\leq N}\overline{c_{\ell}}c_{k} {T_{(k-\ell)}}+(A-I)\sum\limits_{\ell=0}^{N} \vert c_{\ell}\vert^2    \\
&=\sum\limits_{\ell=0}^{N} \vert c_{\ell}\vert^2A+ \sum_{0\leq l\neq k\leq N}\overline{c_{\ell}}c_{k}{T_{(k-\ell)}}\\
&=A^{\frac{1}{2}}\big(\sum\limits_{\ell=0}^{N} \vert c_{\ell}\vert^2 I+ \sum_{0\leq \ell\neq k\leq N}\overline{c_{\ell}}c_{k} A^{-\frac{1}{2}}{T_{(k-\ell)}} A^{-\frac{1}{2}}\big)A^{\frac{1}{2}}\\
&=A^{\frac{1}{2}}\Big
(\sum_{\ell,k=0}^{N}\overline{c_{\ell}}c_{k}\zeta_{A}(k-\ell)\Big)A^{\frac{1}{2}}.
\end{align*}
As $T\in \mathcal{C}_{A},$ by Theorem \ref{Thm: CA class
equivalent},  we have
$\sum\limits_{\ell,k=0}^{N}\overline{c_{\ell}}c_{k}\zeta_{A}(k-\ell)\geq
0.$ Therefore $\varphi_{A}(f(z))\geq 0$ whenever $f(z)$ is strictly
positively element in $\mathcal{S}.$ Now, let $g\in \mathcal{S}$ be
positive. Then $g+\varepsilon \cdot 1$ is strictly positive for any
 $\epsilon >0.$ Then $\varphi_{A}(g)+ \varepsilon A =\varphi_{A}(g+ \varepsilon \cdot 1)\geq 0$ for all $\varepsilon>0.$ Taking limit $\epsilon\to 0,$ we obtain $\varphi_{A}(g)\geq 0.$ Hence the map $\varphi_{A}$ is positive.

 Conversely, let $\varphi_{A}$ be a positive map. In view of Theorem \ref{Thm: CA class equivalent}, it is enough to prove that $$\sum\limits_{\ell,k=0}^{N}\overline{c_{\ell}}c_{k}\zeta_{A}(k-\ell)\geq 0.$$
 Notice that
$$ \sum\limits_{\ell,k=0}^{N}\overline{c_{\ell}}c_{k}\zeta_{A}(k-l)= A^{-\frac{1}{2}} \varphi_{A}\Big(\sum\limits_{0\leq \ell,k\leq N} \overline{c_{\ell}}c_{k} e^{i(k-\ell)\theta}\Big)A^{-\frac{1}{2}}.$$
Since $\varphi_{A}$ is positive and $\sum\limits_{0\leq \ell,k\leq
N} \overline{c_{\ell}}c_{k} e^{i(k-\ell)\theta}\geq 0,$ we have that $$\sum\limits_{\ell,k=0}^{N}\overline{c_{\ell}}c_{k}\zeta_{A}(k-l) \geq 0.$$ This completes the proof.
  \end{proof}
\textcolor{black}{ As an immediate application of  Theorem \ref{Theorem: C_A class and CP
map},  we recover a result of V. Istr\u{a}\c{t}escu
\cite{Istratescu68} that if $0 \leq B\leq A,$ and $A, B$ are
invertible, then $\mathcal{C}_{B}\subseteq \mathcal{C}_{A}.$}
\begin{thm}\label{Prof: commutant of A in CA class}
 Let $A, T\in \mathcal{B(H)}$ such that $TA=AT.$ Then following are \textcolor{black}{equivalent:}
 \begin{enumerate}
\item  $T\in \mathcal{C}_{A}.$
\item For all $ z\in \mathbb{D},$ the following condition holds:
$$A -(A-I)zT-(A-I)\overline{z}T^* + (A-2I)\vert
z\vert^2T^*T\geq 0.$$
\item For all $ z\in \mathbb{D},$ the following condition holds:
 \begin{align*}
 \vert z\vert^2 T^*T\leq \big(A-(A-I)zT\big)^*\big(A-(A-I)zT\big).
 \end{align*}
 \end{enumerate}
\end{thm}

\begin{proof}
$(1) \iff (2)$
Using (iv) of Theorem \ref{Thm: CA class equivalent}, we notice that $T\in \mathcal{C}_{A}$ if and only if
$$(I-zT)^*(A-2I)(I-zT)+ (I-zT)+ (I-zT)^*\geq 0$$ for all $z\in \mathbb{D}.$ Since $TA=AT,$ the result follows immediately.

$(2)\iff (3)$
First, we observe the following computation.
\begin{align*}
 &A^{-\frac{1}{2}}\Big(A -(A-I)zT-(A-I)\overline{z}T^* + (A-2I)\vert
z\vert^2T^*T\Big)A^{-\frac{1}{2}}\\
=&I- A^{-1}(A-I)zT-A^{-1}(A-1)\overline{z}T^* + A^{-1}(A-2I)\vert z\vert^2 T^*T\\
=& \Big(I-A^{-1}(A-I)zT\Big)^{*}\Big(I-A^{-1}(A-I)zT\Big)- A^{-2}\vert z\vert^2T^*T\\
=& A^{-1}\Big(A-(A-I)zT\big)^{*}\big(A-(A-I)zT\Big)A^{-1}-A^{-2}\vert z\vert^2T^*T\\
=& A^{-1} \Big(\big(A-(A-I)zT\big)^{*}\big(A-(A-I)zT\big)-\vert z\vert^2T^*T\Big)A^{-1}.
\end{align*}
It follows that
\begin{align*}
A -(A-I)zT-(A-I)\overline{z}T^* + (A-2I)\vert
z\vert^2T^*T\geq 0 \text{ for all } z\in \mathbb{D}.
\end{align*}
if and only if
 \begin{align*}
 \vert z\vert^2 T^*T\leq \big(A-(A-I)zT\big)^*\big(A-(A-I)zT\big) \text{ for all } z\in \mathbb{D}.
 \end{align*}
\end{proof}
\begin{rmk}
  \textcolor{black}{From }Theorem \ref{Prof: commutant of A in CA class}, we may recover an important characterization (see \cite{DW96} for details) of \textcolor{black}{numerical range contractions, namely that}
$T\in \mathcal{C}_{2}$ if and only if $${\rm Re}(e^{i \theta}T) \leq I, \; \textcolor{black}{\text{for all}\; \theta}.$$
\end{rmk}
\section{Concrete isometric  and unitary dilations for a subclass of $\mathcal{C}_A$-class operators}

Let $A\in \mathcal{B(H)}$ be a positive invertible operator. Recall
that an operator $T$ is in $\mathcal{C}_A$-class means that the
sequence $\{A_{n}\}_{n \geq 1}$ given by
\begin{align}
A_{n}:= A^{-\frac{1}{2}} T^{n}A^{-\frac{1}{2}}, n\geq 1
\end{align}
admits unitary dilation. We do not know how to write down the
dilation of operators of the $\mathcal{C}_A$-class in general. Here, we do it for
a subclass.

First, we write an isometric dilation for 
$\{A^{-\frac{1}{2}}T^nA^{-\frac{1}{2}}\}_{n\geq 1}$. Let $A, C\in
\mathcal{B(H)}$  be \textcolor{black}{a commuting pair} such that $A\geq 0$  with $\Vert
C\Vert \leq 1.$ \textcolor{black}{Then $I+ A(A-2I)C^*C\geq 0$ as $C$ is a contraction. Assume that $$I+
A(A-2I)C^*C= I-C^*C+ (A-I)^2C^*C$$ is invertible and take 
$$T= A[I+ A(A-2I)C^*C]^{-\frac{1}{2}}(I-C^*C)^{\frac{1}{2}}C.$$
 In such a case, we will see that $T$ is in
$\mathcal{C}_A$-class and we can explicitly write down isometric and
unitary dilations of $T$.}

Consider $A,C$ as above. It is convenient to have some notation.
Take $$ B= (I+ A(A-2I)C^*C)^{-\frac{1}{2}}, D=
(I-C^*C)^{\frac{1}{2}}, D_{*}= (I-CC^*)^\frac{1}{2}.$$ Therefore,
$$T=ABDC.$$ Now we take the following sequence $\{T_{n}\}_{n\geq 0} $
of bounded operators defined by
 \begin{align}\label{Eq: CA class}
 T_{0} = I,\; \text{and}\; T_{n}=A^{-\frac{1}{2}} T^{n}A^{-\frac{1}{2}}= A^{n-1}(BDC)^n \text{ for all } n \geq 1.
 \end{align}
Our aim is to show that the sequence $\{T_{n}\}_{n\geq 0}$ admits an
isometric dilation. In other words, there exists an isometry $V$ on
some Hilbert space $\mathcal{K}\supseteq  \mathcal{H}$ such that
$$T_{n}= P_{\mathcal{H}} V^{n}|_{\mathcal{H}}\text{ for all } n\geq 0.$$ In
addition, we want to find a $V$ with explicit block structure.

We observe that since $A$ is positive and commutes with $C$, it
commutes also with $C^*$. It follows that $A$ commutes with also $B, D,
D_*$. Consequently, we get $BD=DB$. On the other hand, we see that
\begin{equation*}
D^{2}_{\ast}C = (I-CC^{\ast})C = C(I-C^{\ast}C) = CD^{2}.
\end{equation*}
As both $D$ and $D_{\ast}$ are positive operators, $D_{\ast}C = CD.$
Furthermore,
\begin{align*}
B^2\big(&D^2+ (A-I)^2C^*C\big) \\
&= \big(I+A(A-2I)C^{\ast}C\big)^{-1}\big( I-C^{\ast}C + A(A-I)C^{\ast}C- (A-I)C^{\ast}C\big)\\
&=  \big(I+A(A-2I)C^{\ast}C \big)^{-1} \big( I-A(A-I)C^{\ast}C-AC^{\ast}C\big)\\
&= (I+A(A-2I)C^{\ast}C)^{-1} (I+A(A-2I)C^{\ast}C)\\
&= I.
\end{align*}
Therefore,
\begin{align}\label{Eq: BDC expression}
(BD)^2=I- B^2(A-I)^2C^*C.
\end{align}
We first obtain a partial isometric dilation on $\mathcal{H}\oplus
\mathcal{H}.$
\begin{lemma}\label{basic lemma}
Define $R$  on $\mathcal{H}\oplus \mathcal{H}$ by:
\begin{align}
R=
\begin{bmatrix}
BDC       &BDD_{*}\\
(A-I)CBC  &(A-I)CBD_{*}
\end{bmatrix}
\end{align}
Then $R$ is a partial isometry and
$$T_{n}= P_{\mathcal{H}}R^{n}P_{\mathcal{H}} \text{ for all } n\geq 1.$$
\end{lemma}
\begin{proof}

Making use of the commutation relations observed above, by direct
computation we see that
$$R^*R=
\begin{bmatrix}
C^*C &C^*D_{*}\\
D_{*}C   &D^2_{*}
\end{bmatrix}.
$$
Clearly, $R^*R$ is a self-adjoint operator and
\begin{align*}
  (R^*R)^2&=   \begin{bmatrix}
C^*C &C^*D_{*}\\
D_{*}C   &D^2_{*}
\end{bmatrix}\begin{bmatrix}
C^*C &C^*D_{*}\\
D_{*}C   &D^2_{*}
\end{bmatrix}\\
     &=
     \begin{bmatrix}
         (C^*C)^2+ C^*D^2_{*}C            &C^*CC^*D_{*}+ C^*D_{*}D_{*}^2\\
         D_{*}CC^*C+D^2_{*}D_{*}C   &D_{*}CC^*D_{*}+D_{*}^4
     \end{bmatrix}\\
    &= \begin{bmatrix}
         C^*C+ C^*(1-CC^*)C            &C^*CC^*D_{*}+ C^*(1-CC^*)D_{*}\\
         D_{*}CC^*C+D_{*}(1-CC^*)C   &D_{*}^2CC^*+D_{*}^2(1-CC^*)
     \end{bmatrix}\\
     &=
      \begin{bmatrix}
C^*C &C^*D_{*}\\
D_{*}C   &D^2_{*}
\end{bmatrix}.\end{align*}
Therefore $R^*R$ is a projection and $R$ is a partial isometry. Next,
we compute $R^2$ as follows:
\begin{align*}
R^2&= \begin{bmatrix}
BDC       &BDD_{*}\\
(A-1)CBC  &(A-1)CBD_{*}
\end{bmatrix}
\begin{bmatrix}
BDC       &BDD_{*}\\
(A-1)CBC  &(A-1)CBD_{*}
\end{bmatrix}\\
&= \begin{bmatrix}
(BDC)^2 +(A-1)BDD_{*}CBC        &*\\
(A-1)CBCBDC +(A-1)^2CBD_{*}CBC &*
\end{bmatrix}\\
&=\begin{bmatrix}
(BDC)^2 +(A-1)BDCDCBC        &*\\
(A-1)CBCBDC +(A-1)^2CBCDBC &*
\end{bmatrix}\\
&=\begin{bmatrix}
A(BDC)^2         &*\\
A(A-1)CBCBDC  &*
\end{bmatrix}.
\end{align*}
Therefore, $P_{\mathcal{H}}R^2|P_{\mathcal{H}}= A(BDC)^2=
A^{-\frac{1}{2}}T^{2}A^{-\frac{1}{2}}.$ Hence by induction, we conclude
that
$$P_{\mathcal{H}}R^n|_{\mathcal{H}}=A^{-\frac{1}{2}}T^{n}A^{-\frac{1}{2}}.$$
\end{proof}
Let $D_{R}=(1-R^*R)^{\frac{1}{2}}$ be the defect operator \textcolor{black}{associated to}
the partial isometry. Since $(1-R^*R)$ is a projection,
$(1-R^*R)^{\frac{1}{2}}=(1-R^*R).$ Therefore, $D_{R}=
\begin{bmatrix}
    D^2 &-C^*D_{*}\\
    -D_{*}C& CC^*
\end{bmatrix}.$
\begin{rmk}\label{Rmk: from Schaffer construction} Using the Sch\"{a}ffer construction of
 $R,$ and Proposition \ref{basic lemma}, we can provide a direct
 explicit form of isometric dilation of the moment sequence $\{A^{-\frac{1}{2}}T^{n}A^{-\frac{1}{2}}\}$
  associated  to the $\mathcal{C}_{A}$-class.
 Let $\tilde{V}: \ell^{2}(\mathcal{H}\oplus \mathcal{H})\to  \ell^{2}(\mathcal{H}\oplus \mathcal{H})$ be the isometry given by
 \[
\tilde{V}= \left[\begin{array}{@{}c|ccccccccc@{}}
     R &0    &0    &0 &0    & \hdots\\
    D_R &0    &0    &0&0    & \hdots\\
\hline
  0                                            &I                   &0               &0           &0    & \hdots\\

  0                                            &0                   &I               &0          &0    & \hdots\\
  0                                            &0                   &0               & I          &0     &\hdots\\
  \vdots                           & \vdots          &\vdots       &\vdots         &\ddots& \ddots\\
\end{array}\right],
\]
Recalling  Sch\"{a}ffer construction of $R$, it is clear that
$P_{\mathcal{H}}\tilde{V}^n|_{\mathcal{H}}=
P_{\mathcal{H}}R^{n}|_{\mathcal{H}}=A^{-\frac{1}{2}}T^{n}A^{-\frac{1}{2}}.$
\end{rmk}
\noindent Here is an alternative construction of an isometric dilation that is simpler to look at.
\begin{thm} \label{thm: explicit isometric A-dilation}
Take $\mathcal{K}= \mathcal{H}^{\oplus 3}\oplus
\ell^{2}(\mathcal{H}).$ Consider $V$ on $\mathcal{K}$ defined by
\[
V= \left[\begin{array}{@{}ccc|cccccccc@{}}
     BDC         & BDD_{*} &0    &0    &0 &0    & \hdots\\
    (A-I)CBC         &(A-I)CBD_{*} &0    &0    &0&0    & \hdots\\
    D               &-C^*                  &0    &0      &0 &0    & \hdots\\
\hline
  0                       &0                      &I                   &0               &0           &0    & \hdots\\

  0                      &0                      &0                   &I               &0          &0    & \hdots\\
  0                      &0                      &0                   &0               & I          &0     &\hdots\\
  \vdots               &  \vdots            & \vdots          &\vdots       &\vdots         &\ddots& \ddots\\
\end{array}\right].
\]
Then $V$ is an isometric dilation of $\{A_n\}_{ n\geq 0}.$
\end{thm}

\begin{proof} The dilation property follows from the Lemma
\ref{basic lemma}. The isometric property of $V$ is clear once we
observe that
$$D_R= (I-R^*R)= \left[ \begin{array}{c}D\\-C\end{array}\right] \left[\begin{array}{cc}
D& -C^*\end{array}\right].$$
\end{proof}
In the last theorem there is no claim of minimality. However, Theorem \ref{thm: explicit isometric A-dilation} \textcolor{black}{provides one step more than} Remark \ref{Rmk: from Schaffer construction} to find minimal dilation.
To obtain the
minimal isometric dilation we need to identify  the minimal dilation
space $\bigvee\limits_{n=0}^{\infty} V^{n}\mathcal{H}.$ This we do
here as a  next proposition using Theorem \ref{thm: explicit isometric A-dilation}.
\begin{prop}\label{prop: explicit minimal dilation space} Let $V$ be the minimal isometric dilation defined in
 Theorem \ref{thm: explicit isometric A-dilation}. Then the minimal dilation space $\bigvee\limits_{n=0}^{\infty}V^n\mathcal{H}$ is of the form
$$
\bigvee\limits_{n=0}^{\infty}V^n\mathcal{H}
=\bigoplus_{n=0}^{\infty} \mathcal{H}_{n},
$$
where $\mathcal{H}_{0}: = \mathcal{H}, \mathcal{H}_{1}:=
V\mathcal{H}\ominus \mathcal{H}.$ Moreover,
$$\mathcal{H}_{1}=\{ ( 0\oplus (A-I)CBCh \oplus Dh \oplus \cdots ):
h\in \mathcal{H}\}$$ and
\begin{align*} \mathcal{H}_{n}&:= V^n\mathcal{H}\ominus \bigvee
\limits_{m=0}^{n-1} V^m\mathcal{H}\\
& = \Big\{ (0\oplus  (n ~\mbox{copies}~)
\oplus 0 \oplus (I-A)BCPh\oplus Dh \oplus 0 \cdots ):
h\in\mathcal{H}\Big\}
\end{align*}
for $n\geq 2,$ where $P$ is the projection onto
$\ker(I-A)BDC.$
\end{prop}
\begin{proof} First note that
$\mathcal{H}\bigvee V\mathcal{H}= \mathcal{H}\oplus (V\mathcal{H}\ominus \mathcal{H}).$
Let $h\in \mathcal{H},$ then notice that
\begin{align*}
Vh&=\big( BDCh \oplus (A-I)CBCh
\oplus Dh\oplus 0\oplus \cdots \big),\\
V^2h&=\Big( A(BDC)^2h\oplus
A(A-I)CBCBDCh\oplus DBDCh-(A-I)C^*CBCh\\
& \;\;\; ~~~~~~~~~~~~~ \oplus Dh\oplus 0\oplus \cdots
\Big).\end{align*}
Then it is immediate to see that
$$V\mathcal{H}\ominus \mathcal{H}=\{ (
0\oplus (A-I)CBCh\oplus Dh\oplus 0\oplus 0\oplus \cdots ):  h\in
\mathcal{H}  \}.$$
 Next, we shall  find $V^2\mathcal{H}\ominus
(\mathcal{H}\bigvee V\mathcal{H}).$ First, we write $V^2h$ as
 \begin{align*}
V^2h &= \big( A(BDC)^2h\oplus 0\oplus \cdots\big )
+ \big( 0\oplus
(A-I)CBC(ABDCh)\\
& \;\;\;~~~~~~~~~~ \oplus DABDCh\oplus 0\oplus \cdots \big) +\big( 0\oplus
0\oplus (I-A)BCh\oplus Dh\oplus 0\oplus  \cdots \big).\end{align*}
 Obviously, the
first and second  terms  belong to $\mathcal{H}$
 and
$V\mathcal{H}\ominus \mathcal{H}$ respectively.  We want to
understand the last term of this equation.  We claim that
$$
V^2\mathcal{H}\ominus (\mathcal{H}\vee V\mathcal{H})= \big\{ \big(0\oplus
0\oplus (I-A)BCPh\oplus Dh\oplus 0\oplus \cdots \big) : h\in\mathcal{H}
\big\}.$$
 Notice that $$\big\langle \big( 0\oplus (A-I)CBCg\oplus Dg\oplus 0\oplus
0\oplus \cdots \big), \big( 0\oplus 0\oplus (I-A)BCPh\oplus Dh\oplus 0\oplus
\cdots \big)\big\rangle =0$$ as $(I-A)BDCPh=0.$  Therefore, we have $$
\big\{\big(0\oplus 0\oplus (I-A)BCPh\oplus Dh\oplus 0\oplus \cdots \big) :
h\in\mathcal{H}\big \}\subseteq V^2\mathcal{H}\ominus (\mathcal{H}\vee
V\mathcal{H}).$$
Now, it is sufficient to prove that
$$ \big(0\oplus 0\oplus
(I-A)BC(I-P)h\oplus 0\oplus \cdots \big) \in V\mathcal{H}\ominus
\mathcal{H} $$ for all $h \in \mathcal{H}.$ To see this, suppose
\begin{align} \label{Equ: Essential equ}
\Big\langle \big( 0\oplus 0\oplus (I-A)BC(I-P)h\oplus &0\oplus \cdots \big), \nonumber \\
&\big(0\oplus (A-I)CBCg\oplus Dg\oplus 0\oplus \cdots \big) \Big\rangle =0,
\end{align}
for all $ g\in \mathcal{H}.$  This implies that for all $g\in
\mathcal{H},$
\begin{align*}
0&=\big \langle (I-A)BC(I-P)h, Dg \big\rangle \\
&= \big \langle(I-A)BDC(I-P)h, g \big\rangle\\
 &=  \big \langle (I-A)BDCh, g \big\rangle ,
 \end{align*}
as $(I-A)BDC(I-P)h=0.$
 It follows that $$(I-A)BDCh=0 \text{ and } h\in
{\rm ran}{P}.$$
Therefore, we have $$ \big(0\oplus 0\oplus
(I-A)BC(I-P)h\oplus 0\oplus \cdots\big )  =0$$ under the condition
defined Equation \eqref{Equ: Essential equ}. Consequently, we get $$\big(
0\oplus 0\oplus (I-A)BC(I-P)h\oplus 0\oplus \cdots \big)  \in
V\mathcal{H}\ominus \mathcal{H}$$ for all $h\in \mathcal{H}.$ Hence,
we have \begin{align*}
&\big \{ \big( 0\oplus 0\oplus (I-A)BCPh\oplus Dh\oplus 0\oplus
\cdots \big ) : h\in\mathcal{H} \big\}\\
=&V^2\mathcal{H}\ominus (\mathcal{H}\vee
V\mathcal{H}).\end{align*}
Now it is easy to see that
\begin{align*}
&\big\{\big (
0\oplus 0\oplus 0\oplus (I-A)BCPh\oplus Dh\oplus 0\oplus \cdots \big ) :
h\in\mathcal{H} \big\}\\
=&V^3\mathcal{H}\ominus (\mathcal{H}\vee
V\mathcal{H} \vee V^2\mathcal{H}).
\end{align*}
By induction, we can prove
that
\begin{align*}
&\big\{\big (
0\oplus ( n~\mbox{copies}~ ) \oplus (I-A)BCPh\oplus Dh\oplus 0\oplus
\cdots \big ): h\in\mathcal{H} \big\}\\
=&V^n\mathcal{H}\ominus (\bigvee\limits_{m=0}^{n-1} V^m\mathcal{H}).
\end{align*}
\end{proof}

Now we explicitly write down a unitary dilation for the sequence
$\{T_{n}\}_{n \geq 0}.$

\begin{thm} \label{thm: explicit unitary A-dilation}
Let $\mathcal{K}= \ell^{2}(\mathcal{H})\oplus \mathcal{H}^{\oplus3}\oplus \ell^{2}(\mathcal{H})$ and  $U$  acting on $\mathcal{K}$ is defined by the following block matrix representation
\begin{align}\label{Eq: Explicit unitary of CA}
\left[\begin{array}{@{}ccc|cccc|cccc@{}}
\ddots &                     &                    &                             &                             &                                               &      &      &           &\\
          &I                     &                    &                             &                             &                                              &      &      &           &\\
          \hline
          &                      &I                   & 0                            &  0                           &0                                        & 0     &      &           &\\
          &                      &                    &-(A-I) BC^*           &{\bf BDC}               & BDD_{*}                               &0      &      &            &\\
          &                      &                    & B_{*}D_{*}          & (A-I)CBC              &  (A-I)CBD_{*}                        &0      &      &             &\\
          &                      &                    & 0                         &D                           &-C^*                                       & 0     &      &            &\\
\hline
          &                      &                    & 0                        &0                            & 0                                           &I   &       &            &\\
          &                      &                    &                            &                               &                                             &    & I    &             &\\
          &                      &                    &                            &                               &                                             &     &     &\ddots    &\\
          \end{array}\right].
\end{align}
where $BDC$ that appear in the bold font is the $(00)$ entry
of $U,$ and the  $B_{*}$ is given by
 \begin{equation*}
 B_{*}= (I+ A(A-2I)CC^*)^{-\frac{1}{2}}.
 \end{equation*}
Then $U$ is a unitary dilation of $\{T_n\}_{n\geq 0}.$
\end{thm}
\begin{proof} Consider the block operator matrix $M$ defined by:
\begin{align}
M= \begin{bmatrix}
&0                    &0&0&0\\
&-(A-I) BC^*    &{\bf BDC}& BDD_{*}&0\\
& B_{*}D_{*}    &(A-I)CBC&(A-I)CBD_{*}&0\\
&0                   &D&-C^*&0
\end{bmatrix}.
\end{align}
It is straight forward to verify that $$M^*M=
\begin{bmatrix}
I&0&0&0\\
0&I&0&0\\
0&0&I&0\\
0&0&0&0\\
\end{bmatrix}, ~~~~MM^*=
\begin{bmatrix}
0&0&0&0\\
0&I&0&0\\
0&0&I&0\\
0&0&0&I\\
\end{bmatrix}.$$ This clearly implies that  $U^*U= UU^*=I.$ Finally, we obtain the desired claim
\begin{align*}
P_{\mathcal{H}}U^{n}|_{\mathcal{H}}= P_{\mathcal{H}}M^{n}|_{\mathcal{H}} = T_{n} \text{ for all } n
\geq 0.
\end{align*}
This completes the proof.
\end{proof}
 As we discussed earlier, minimal dilation space of the unitary dilation is of the form $$\mathcal{K}=\bigvee^{\infty}\limits_{n=-\infty} U^{n} \mathcal{H}.$$ The minimal dilation space can be written as
  $$\mathcal{K}_{+}\bigvee \mathcal{K}^{-}=\mathcal{K},
\text{ where }  \mathcal{K}_{+}=\bigvee^{\infty}\limits_{n=0} U^{n}
\mathcal{H},$$ and $\mathcal{K}^{-}=\bigvee^{-\infty}\limits_{n=0}
U^{n} \mathcal{H}.$ Since $U|_{\mathcal{K}_+}$ is the isometry $V$
constructed before, the decomposition of $\mathcal{K}_+$ is clear.

\begin{rmk}\label{rmk: explicit minimal dilation space}  $\mathcal{K}_{-}$ can be decomposed as
$$
\mathcal{K}_{-}
=\bigoplus_{n=1}^{\infty} \mathcal{H}_{-n}
$$
$\text{ where } \mathcal{H}_{-1}:= U^*\mathcal{H}\ominus
\mathcal{H}, \mathcal{H}_{-n}:= U^{*n}\mathcal{H}\ominus \bigvee
\limits_{m=0}^{n-1} U^{*m}\mathcal{H}   \text{ for } n \geq 2.$
Moreover,  \begin{align*}
     \mathcal{H}_{-1}&:= \big\{ \big( \cdots\oplus
-(A-I)CBh\oplus \textbf{0}\oplus D_{*}DBh\oplus 0\oplus \cdots \big ):
h\in \mathcal{H} \big \},
\end{align*}
\begin{align*}
 \mathcal{H}_{-n} := \Big\{\big( \cdots\oplus
-(A-I)&CBCh\oplus B^{-\frac{1}{2}}_{*}DBQh\\
& \;\;\;\;\;\ \oplus 0 \oplus (-n+2)
~\mbox{times}~\oplus \textbf{0}\oplus \cdots \big): h\in\mathcal{H} \Big\}
\end{align*}
and $Q$ is the projection onto $\ker(A-1)B^{\frac{1}{2}}C^*DB.$
\end{rmk}

\subsection{Special case ($\mathcal{C}_{\rho}$-class)} The case when $A$ is a positive scalar $\rho,$ corresponds to the $\mathcal{C}_{\rho }$-
class of operators. This class is very rich and widely studied. However, the explicit block operator description of isometric and unitary dilations of operators \textcolor{black}{in this class} is unknown in the
literature except for $\rho =2$. An operator $T$ is in the
$\mathcal{C}_2$ class if and only if $w(T)\leq 1,$ where $w(T)$ denotes the
numerical range of $T$. The unitary dilation of this class was exhibited
by T. Ando \cite{Ando 73} and coincides with the following
construction (with $\rho =2$).
We recall from Durzst \cite{Durszt75} that $T\in \mathcal{C}_{\rho}
(\rho
>0)$ if and  only if $T$ is of the form:
$$
T= \rho(1+ \rho(\rho-2)C^*C)^{-\frac{1}{2}}DC,
$$
where  $\rho>0, C $ is a contraction and $D=(I-C^*C)^{\frac{1}{2}}.$
 In this case, we may write down isometric and unitary dilations as above. We observe that $$B=(I +
\rho(\rho-2)C^*C)^{-\frac{1}{2}}, B_{*}=(I +
\rho(\rho-2)CC^*)^{-\frac{1}{2}}$$ and $$T_{n}=
\rho^{n-1}(BDC)^n=\frac{(\rho BDC)^n}{\rho}=\frac{T^{n}}{\rho}.$$

\begin{rmk}
    Moreover, for the case $\rho=2,$ being $B=B^*=I,$ the
    operators $B, B^*$ will not play any role in the unitary and isometric dilation of the $\mathcal{C}_{2}$-class.
\end{rmk}
 \begin{rmk}Lemma \ref{basic lemma} yields a crucial observation. If $T \in \mathcal{C_{\rho}},$ then $\frac{T^{*n}}{\rho}$ admits a partial isometric dilation. By utilizing this partial isometric dilation, \textcolor{black}{one can understand the dilation of such class operators completely}. It is worth noting that this observation was previously unknown even in the case when $\rho=2$. For further details on Ando's dilation theorem for $\rho=2$, we recommend referring to \cite{Ando 73} (also see \cite{DW96}).

 \end{rmk}
\begin{Question} How to write block decompositions for isometric and
unitary dilations of general operators of the $\mathcal{C}_A$ class?
\end{Question}



\subsection*{Acknowledgment}
The first named author is funded by
the J C Bose Fellowship JBR/2021/000024 of SERB(India). The second named
author is supported by the NBHM postdoctoral fellowship, Department
of Atomic Energy (DAE), Government of India (File No. 0204/1(4)/2022/ R\&D-II/1198). The third named author is funded by the Startup Research Grant (File No. SRG/2022/001795) of SERB, India.
The authors also acknowledge the Stat-Math Unit of the Indian Statistical Institute Bangalore Centre for providing an excellent research environment. The second named author also gained valuable insights from several informative discussions with Tirthankar Bhattacharyya. 

We express our sincere gratitude to the anonymous referee for providing valuable suggestions that helped improve the presentation of our manuscript.

\end{document}